\documentclass[11pt]{amsart}

\usepackage{graphicx}
\usepackage{bm}
\usepackage{pifont}
\usepackage{adjustbox}
\usepackage{tikz}
\usepackage{color}
\usepackage{hyperref}
\usepackage{amsmath, amssymb}

\newcommand{\la}{\langle}
\newcommand{\ra}{\rangle}
\newcommand{\x}{\times}

\newcommand{\too}{\longrightarrow}
\newcommand{\bd}{\partial}

\newcommand{\cC}{{\mathcal{C}}}
\newcommand{\cD}{{\mathcal{D}}}

\newcommand{\cL}{{\mathcal{L}}}
\newcommand{\cO}{\mathcal O}
\newcommand{\cA}{\mathcal{A}}

\newcommand{\cU}{\mathcal{U}}

\newcommand{\cP}{\mathcal{P}}
\newcommand{\NN}{\mathbb{N}}
\newcommand{\ZZ}{\mathbb{Z}}
\newcommand{\CC}{\mathbb{C}}

\newcommand{\CP}{\mathbb{CP}}

\newcommand{\QQ}{\mathbb{Q}}

\DeclareMathOperator{\lcm}{lcm}

\renewcommand{\a}{\alpha}
\renewcommand{\b}{\beta}
\renewcommand{\d}{\delta}
\newcommand{\g}{\gamma}
\newcommand{\e}{\varepsilon}
\newcommand{\eps}{\epsilon}
\newcommand{\f}{\varphi}

\renewcommand{\l}{\lambda}
\renewcommand{\k}{\kappa}
\newcommand{\s}{\sigma}
\newcommand{\m}{\mu}
\newcommand{\n}{\nu}
\renewcommand{\o}{\omega}

\renewcommand{\t}{\tau}

\newcommand{\G}{\Gamma}

\renewcommand{\S}{\Sigma}

\newcommand{\orb}{\mathrm{orb}}

\newcommand{\ii}{\mathrm{i}}

\newtheorem{theorem}{Theorem}[section]
\newtheorem{corollary}[theorem]{Corollary}
\newtheorem{proposition}[theorem]{Proposition}

\newtheorem{lemma}[theorem]{Lemma}
\newtheorem*{theorem*}{Theorem}

\theoremstyle{definition}
\newtheorem{definition}[theorem]{Definition}

\newtheorem{remark}[theorem]{Remark}

\begin{document}

\title[On K-contact non-Sasakian Smale-Barden manifolds]{On the construction of K-contact 
non-Sasakian \\ Smale-Barden manifolds}

\author{Vicente Mu\~noz}
\address{Departamento de Algebra, Geometr\'{\i}a y Topolog\'{\i}a, Universidad Complutense de Madrid, Plaza de Ciencias, Ciudad Universitaria, 28040 Madrid, Spain.} 
\email{vicente.munoz@ucm.es}

\author{Juan Rojo}
\address{ETS Ingenieros Inform\'aticos, Universidad Polit\'ecnica de Madrid, Campus de Montegancedo, 28660 Madrid, Spain.}
\email{juan.rojo.carulli@upm.es}
%
%
\maketitle

\begin{abstract}
    In the breakthrough paper \cite{Mu-jems}, it 
is constructed the first example of a simply connected compact 
$5$-manifold (aka.\ Smale-Barden manifold) which admits a K-contact structure but does not carry a Sasakian structure.
In this work we clarify some aspects of the construction of \cite{Mu-jems}, determining explicitly the number $N$ of symplectic surfaces needed to 
have an isotropy locus that produce a $5$-manifold that is K-contact but not Sasakian.
Also, in order to analyse the geography problem of determining which 
Smale-Barden manifolds admit K-contact but not Sasakian structures, we 
refine and generalize the constructions of symplectic surfaces in a symplectic $4$-manifold with transversal intersections giving rise to such manifolds.
\end{abstract}

\section{Introduction}
\label{sec:1}

A central question in geometry is to determine when a given manifold admits a specific geometric structure. 
For instance, it is already a classic problem the question of constructing (compact) symplectic manifolds that do not admit K\"ahler structures.
In odd dimensions, the analogues of K\"ahler and symplectic manifolds are Sasakian and K-contact manifolds, so a natural question asks for the existence of K-contact manifolds which do not admit Sasakian structures.
In higher dimensions, the problem is solved by means of topological obstructions, like the parity of $b_1$, the rational homotopy type or the hard Lefschetz property. 
The problem of the existence of simply connected K-contact non-Sasakian compact manifolds in dimension $5$ is the hardest and a cornerstone in this area. It appears as Open Problem 10.2.1 in the treatise \cite{BG}:

\smallskip

{\it Are there Smale-Barden manifolds with K-contact but not Sasakian structures?}

\medskip

A simply connected compact $5$-manifold is called a {\it Smale-Barden manifold}. These manifolds are classified \cite{B,S} by their second homology group 
   \begin{equation*} 
  H_2(M,\ZZ)=\ZZ^k\oplus ( \mathop{\oplus}\limits_{p,i}\ZZ_{p^i}^{c(p^i)}),
  \end{equation*}
where $k=b_2(M)$ and the numbers $p$ are distinct primes, together with the second 
Stiefel-Whitney map $w_2: H_2(M,\ZZ)\rightarrow\ZZ_2$. The fact that we have a classification of all 
Smale-Barden manifolds raises the {\em geography problem} of determining which Smale-Barden manifolds admit K-contact and which admit Sasakian structures.

A Sasakian manifold $M$ always admits a \emph{quasi-regular} Sasakian structure, that is it is a 
Seifert bundle over a cyclic K\"ahler orbifold $X$.
In the case of a $5$-manifold, $X$ is a singular
complex surface with cyclic quotient singularities, and the isotropy locus is formed by complex curves intersecting transversely. 
Analogously, a K-contact manifold admits a quasi-regular K-contact structure, which is a Seifert bundle over an
almost-K\"ahler orbifold
(that is, a symplectic orbifold with a compatible almost complex structure, which always exists) with cyclic singularities, and the isotropy locus is formed by symplectic surfaces, intersecting positively (and only pairwise).

The topology of a simply connected $5$-manifold $M$ which is a Seifert bundle over a cyclic $4$-orbifold $X$ is determined in \cite[Therorem 36]{Mu}. The homology is 
 $$
 H_2(M,\ZZ)=\ZZ^k\oplus (\mathop{\oplus}_i \ZZ_{m_i}^{2g_i}),
 $$
where $k=b_2(M)=b_2(X)-1$, and the isotropy locus has 
isotropy surfaces $D_i$ with multiplicities $m_i$, and genus $g_i=g(D_i)$. We have the conditions $H_1(X,\ZZ)=0$,  
$\gcd(m_i,m_j)=1$ if $D_i,D_j$ intersect, 
and $H^2(X,\ZZ)\to \mathop{\oplus}_i H^2(D_i,\ZZ_{m_i})$ is surjective.
There are some technical conditions on the Chern class of the
Seifert bundle $M\to X$, and on the orbifold fundamental group 
$\pi_1^{\orb}(X)$ to arrange that $M$ is simply-connected, and to characterize the Stiefel-Whiteny class $w_2(M)$.

 To produce K-contact Smale-Barden manifolds, one needs to construct symplectic  $4$-orbifolds with cyclic quotient
 singularities having symplectic surfaces 
 of given genus inside. If the isotropy coefficients are not coprime, these surfaces are forced to be
 disjoint (and linearly independent in homology). In particular,
 the number of such surfaces in the isotropy locus is at most
 $b_2(X)=k+1$. Some properties of the configuration of these isotropy surfaces can be read from  $H_2(M,\ZZ)$, for instance: the genus, the isotropy 
 coefficients, and whether they are disjoint. This fact opens a way to attack the problem of finding a K-contact manifold without Sasakian structure, as configurations of symplectic surfaces are less rigid than configurations of complex curves. This was exploited in \cite{MRT} and \cite{CMRV} to produce K-contact $5$-manifolds with $H_1(M,\ZZ)=0$ and simply-connected, and not admitting
 a \textit{semi-regular} Sasakian structure.
   
The question \cite[Open Problem 10.2.1]{BG} is solved in the breakthrough paper \cite{Mu-jems}:

\begin{theorem}[{\cite{Mu-jems}}] \label{thm:main-jems}
There exists a Smale-Barden manifold $M$ which admits a K-contact structure but does not admit a Sasakian structure.
In concrete terms, there is some $N>0$ large enough, and distinct primes $p_{nm}>\max(3,n,m)$, $1\leq n,m\leq N$, so that
 \begin{equation*}
  H_2(M,\ZZ)= \ZZ^2 \oplus \bigoplus_{n,m=1}^N 
 \left(\ZZ^{18n^2+2}_{p_{nm}}  \oplus \ZZ^{18m^2+2}_{p_{nm}^2} \oplus \ZZ^{20}_{p_{nm}^3}\right) .
 \end{equation*}
\end{theorem}

The construction of $M$ is based on the following.
First, with techniques from symplectic geometry, we construct a symplectic orbifold with $b_2(X)=3=b_2^+(X)$ for which there are $3$ disjoint symplectic surfaces
$T_1,T_1',A$. This is impossible for K\"ahler manifolds since for them 
it always holds $b_2^+=1$. As these have positive self-intersection we can represent positive multiples of the homology classes by symplectic surfaces, so we can form many
tuples of
$3$ disjoint symplectic surfaces $(T_n, T'_m, A_k)$, $n,m,k \in \NN$, and so that
\[
1 \le n,m,k \le N, \quad \text{with $N$ a large number to be determined.}
\]
In fact, in \cite{Mu-jems} only tuples with $k=1$ are considered, as these suffice for the purposes.
For each tuple $(T_n, T'_m, A)$, we put
isotropy coefficients $p_{nm},p_{nm}^2,p_{nm}^3$, with $p_{nm}$ distinct primes, and consider the manifold $M$  obtained as a Seifert bundle 
over such orbifold. The proof that $M$ does not admit a Sasakian structure requires to check that
there does not exists any singular complex surface $Y$ with cyclic singularities with $b_2(Y)=3$, and a large number of tuples $\varepsilon_{nm}
=(D_1^{nm},D_2^{nm},D_3^{nm})$, $1\leq n,m\leq N$, each consisting of $3$ disjoint complex curves, which
will form the isotropy locus of the Seifert bundle $M\to Y$. The genera of these curves are controlled, and two of the curves are negative by the Hodge-Riemann relations. One needs to bound the number of singular points (universally, i.e.\ independently of $Y$). 
As the number of singular points is bounded, one has that 
most of the tuples $\varepsilon_{nm}$
avoid the singular points, and those satisfy the adjunction equality and have self-intersection which is an integer. 
In this way, we get many tuples non-proportional to each other and which do not pass through singularities; if $K_Y$ is the canonical divisor of $Y$, we can write $K_Y^2$ with respect to each basis, obtaining a collection of
diophantine equalities, which become impossible by 
choosing $N$ large enough.

\medskip

The present paper has two main objectives.
The first objective is to generalize the construction of a $K$-contact non-Sasakian manifold. This is done by also taking multiples of the third curve $A$, i.e. by allowing $k$ to also vary in the tuples $(T_n, T'_m, A_k)$. This changes the orbifold structure and the fundamental group, so we need to check that everything works in this case as well. In particular, this requires to check that the orbifold fundamental group $\pi_1^{\orb}(X)$ is abelian; as $X$ has many codimension-2 isotropy surfaces, this is a non-trivial matter, and it is handled in Theorem \ref{thm:pi1(M)}. We conclude the following.

\begin{theorem}\label{th:main-0}
For any numbers $N_1, N_2, K_0 \in \NN$, and for any collection of distinct primes $p_{nmk}$ with  
\[
p_{nmk}>\max(3,n,m,k) , \quad 1 \leq n \leq N_1, 1\leq m \leq N_2, 1\leq k\leq K_0,
\] 
there is a Smale-Barden manifold $M$ carrying a $K$-contact structure and with invariants given by
 \begin{equation*}
  H_2(M,\ZZ)= \ZZ^2 \oplus \bigoplus_{n,m, k=1}^{N_1, N_2, K_0}
 \left(\ZZ^{18n^2+2}_{p_{nmk}}  \oplus \ZZ^{18m^2+2}_{p_{nmk}^2} \oplus \ZZ^{18k^2+2}_{p_{nmk}^3}\right) .
 \end{equation*}
 and with second Stiefel-Withney class $w_2(M)=0$, i.e.\ $M$ is spin.
\end{theorem}

Applying the results from \cite{Mu-jems} to the case $K_0=1$ and $N_1=N_2=N$ in Theorem \ref{th:main-0} above, we immediately get the following result.

\begin{corollary}
\label{cor:main-1}
There is a large number $N$ so that for any $K_0 \ge 1$ and for any $N_1, N_2 \ge N$, 
the Smale-Barden manifolds $M$ from Theorem \ref{th:main-0} do not admit Sasakian structures.
\end{corollary} 

We can get an improvement of this corollary with some work, using the fact that by increasing the number $K_0$ of multiples of the third curve we get better bounds.

\begin{theorem} \label{th:main-1}
Consider $\t_0=10 \, 290$, $\k_0=2 \t_0+2$, $\mu_0=2\k_0$.
There exists a big number $\hat N$ so that for any $N_1 \ge \hat N$, 
$N_2\geq \mu_0$, $K_0\geq \kappa_0$, the $K$-contact Smale-Barden manifolds $M$ constructed in Theorem \ref{th:main-0} do not admit any Sasakian structure.
\end{theorem}

We remark that, while the number $\hat N$ is very large, it is considerably smaller than the number $N$ from Corollary \ref{cor:main-1}. So, at the expense of increasing $K_0$ a bit, we get much better bounds for the bounds of $N_2$ and $N_1$.

\medskip

The second aim of this work is to determine a specific bound for the parameter $N$ that appears in the manifolds obtained in Theorem \ref{thm:main-jems} and for $\hat N$ in Theorem \ref{th:main-1}, thus avoiding the 
wording ``for $N$ large enough''. We obtain the following result.

\begin{theorem}\label{th:main-2}
Let $R=18^2 \lcm(2,3,4,\dots, 2\,548\,292\,959\,346)$.
For any number $N$ such that
\[
N \ge 
\tfrac{16}{17} \left(\tfrac{18}{16}\right)^{2^{65\,499}} (rR)^{2^{65\,500}-1} , 
\]
the $K$-contact Smale-Barden manifolds $M$ given in Theorem \ref{thm:main-jems} cannot admit a Sasakian structure. Also, the $K$-contact Smale-Barden manifolds $M$ given in    Theorem \ref{th:main-0} with $N_1, N_2 \ge N$ cannot admit a Sasakian structure.
\end{theorem}

The computation of this bound for $N$ follows the arguments of \cite[Section 8]{Mu-jems}, by computing explicitly each of the universal constants appearing along the paper.

We also develop an anternative argument to deduce the impossibility of the existence of a Sasakian structure, this time using specifically the multiples of the third curve. This argument lowers the value of $N$, but requires some conditions on the number $K_0$.

\begin{theorem} \label{th:main-3}
Let $\hat R= 18^2 \lcm(2,3,4,\dots, 579\,510\,414\,464)$, and $\hat r=1+10^{-28}$.
Let $\hat N$ an integer such that
\[
\hat N \ge 
\tfrac{16}{17} \left(\tfrac{18}{16}\right)^{2^{20\,582}} (\hat r \hat R)^{2^{20\,583}-1} \, .
\]
The K-contact Smale-Barden manifolds $M$ from Theorem \ref{th:main-0} with $N_1\ge \hat N$,
$N_2 \ge 41\,164$ and $K_0 \ge 20 \, 582$, cannot admit a Sasakian structure.
\end{theorem}

Finally, in Theorems \ref{thm:16} and \ref{thm:17} we give a sharper version of the construction of $>K$-contact non-Sasakian manifolds, and we obtain many more examples of such manifolds. This is achieved by inspecting the essential property that we need in the torsion of $H_2(M,\ZZ)$ in order to ensure the non-existence of a Sasakian structure.

\noindent {\bf Acknowledgements.}
The second author is grateful to Enrique Artal, Eva Elduque and Juan Viu-Sos for useful conversations about fundamental groups of complements of plane curves.

\section{Construction of the $K$-contact manifold}

In this section we partly review the construction of the $K$-contact manifold from \cite{Mu}, with the necessary modifications.
Take the rational elliptic surface $S$ with singular fibers $I_9+3A_1$, that
appears in  \cite[p.\ 568]{Beauville}.
This is a rational elliptic surface, with an
elliptic fibration $S\to \CP^1$ obtained from a 
suitable pencil of cubic curves after blow-ups,
which has four singular fibers; three of them are rational nodal curves (the $A_1$'s), and one
is a cycle of $9$ rational curves $C_1\cup C_2\cup
\ldots \cup C_9$ (the $I_9$ singularity). This fibration has three sections
$\s_1,\s_2,\s_3$ which are rational curves of
self-intersection $\s_i^2=-1$, and 
$\s_1\cdot C_1=1$, $\s_2\cdot C_4=1$,
$\s_3\cdot C_7=1$.

Let $F$ be a smooth fiber of the elliptic fibration $S\to \CP^1$, and fix 
an isomorphism $H_1(F,\ZZ) \cong \ZZ^2$. The
 monodromy of the fibration, No.\ 63 of \cite[Table 3]{Fukae}, is described by the
 equality 
 $X_{[1,1]}X_{[1,-2]}X_{[2,-1]} A^9 = \mathrm{I}$, 
 $A=X_{[1,0]}$. We use the notation
 $X_{[p,q]}=\left(\begin{array}{cc} 1+pq & -p^2 \\ q^2 & 1-pq\end{array}\right)$, which corresponds
 to a Dehn twist along the curve $(p,q) \in\ZZ^2=
 H_1(F,\ZZ)$. 
 The vanishing cycle of $X_{[p,q]}$ is $(p,q)$, with the choice of path to the critical point
taken in \cite{Fukae}. Therefore monodromies corresponding to going around the nodal curves
are $X_{[1,1]}$, $X_{[1,-2]}$ and $X_{[2,-1]}$, whence the vanishing cycles are $(1,1)$, $(1,-2)$ and $(2,-1)$.
Finally, $A^9$ is the monodromy around the $I_9$
 singular curve.

 We take two copies $S_1,S_2$ of $S$ as above, with two smooth fibers $F_1,F_2$. They have 
 self-intersection $F_1^2=F_2^2=0$.
 We choose a symplectomorphism $\varphi:F_1 \to F_2$ such that the vanishing cycles match,
 that is, the identity in homology $\varphi_*:H_2(F_1,\ZZ) \to H_2(F_2,\ZZ)$. 
Take the Gompf symplectic sum 
 \begin{equation*}
 X=S_1 \#_{F_1= F_2} S_2 = (S_1-\nu(F_1))\cup_{\bd \nu(F_1)\cong\bd\nu(F_2)} (S_2-\nu(F_2))\, .
  \end{equation*}
As $b_2(S_1)=b_2(S_2)=10$, then 
$\chi(S_1)=\chi(S_2)=12$. So $\chi(X)=24$ and hence $b_2(X)=22$. It is easy to see that $X$ is simply connected.

Let $C_1,\ldots, C_9$ be the $I_9$-cycle of $S_1$, with sections $\s_1,\s_2,\s_3$ as before.
Let $C'_1,\ldots, C'_9$ be the $I_9$-cycle of $S_2$, with corresponding sections $\s'_1,\s'_2,\s'_3$.
By using \cite[Lemma 24]{MRT}, we can glue the sections to produce symplectic surfaces $E_1,E_2,E_3$
of square $-2$.

 Fix a fiber $F\subset \bd \nu(F_1)=\bd \nu(F_2)\subset X$. 
 Take the vanishing cycle $a=(1,1)$ in $F$, and the two corresponding vanishing thimbles $D_1,D_2$ in $S_1-\nu(F_1), S_2-\nu(F_2)$, respectively.
We glue them to form a Lagrangian $(-2)$-sphere $D$. Next take as dual curve in $F$ the curve
$b=(1,-2)$, intersecting $a$ transversally and positively at three points. Take the torus $T=b\x S^1 \subset
F\x S^1=\bd \nu(F_1)=\bd \nu(F_2)$. This produces a pair of surfaces $D,T$ with 
  \begin{equation}\label{eqn:corr01}
   D^2=-2, \, D\cdot T=3, \,  T^2=0,
   \end{equation}
where $D$ is a Lagrangian sphere and $T$ is a Lagrangian torus.

With the second vanishing cycle $b=(1,-2)$, we do the same thing using another fiber $F'
\subset \bd \nu(F_1)=\bd \nu(F_2)\subset X$. We obtain $D'$ a Lagrangian $(-2)$-sphere by gluing the vanising thimbles associated to $b$, and a Lagrangian
torus $T'=a \times S^1$, with $T'^2=0$, $D'^2=-2$ and $T' \cdot D'=3$.

Finally, we can change slightly the symplectic form so that $D,D'$ become
symplectic $(-2)$-surfaces, and $T,T'$ become symplectic tori by \cite[Lemma 27]{MRT}.
Next, we see that there is a chain of $17$ rational $(-2)$-curves
  \begin{equation*}
  \cC= C_8\cup  \ldots\cup  C_2\cup C_1\cup E_1 \cup  C_1'\cup  C_2'\cup  \ldots\cup  C_8'\, ,
  \end{equation*}
where $E_1$ is the gluing of two sections, so it is a symplectic $(-2)$-sphere intersecting transversely $C_1$ and $C'_1$.

We have the following:
\begin{itemize}
    \item There are neighborhoods $U$ of $\cC$, $V$ of $T \cup D$, and $V'$ of $T' \cup D'$ which are disjoint.
    \item The section $E_2$ is disjoint from $T \cup D$ and $T' \cup D'$, and intersects transversely the chain $\cC$ at two points, one of them in $C_4$ and another in $C'_4$.
    \item The section $E_3$ is disjoint from $T \cup D$ and $T' \cup D'$, and intersects transversely the chain $\cC$ at two points, one of them in $C_7$ and another in $C'_7$.
\end{itemize}

\subsection*{Construction of the orbifold $X'$}

 We contract $D$ and $D'$ to two points $p,p'$ of multiplicity $2$, and the chain $\cC$ to a
 point $q$ of multiplicity $18$. See \cite[Proposition 11]{Mu-jems} for a discussion of contractions of chains of symplectic surfaces. 
 Note that $[2,\stackrel{(17)}{\ldots}, 2]=\frac{18}{17}$, so the point
 has local model $(z_1,z_2)\mapsto (\varepsilon z_1,\varepsilon^{-1}z_2)$, with $\varepsilon=e^{2\pi i/18}$.
 Denote $\bar X$ the resulting symplectic cyclic orbifold with singular set $P=\{p,p',q\}$ and contraction map $\pi: X \to \bar X$. It has
 $b_2(\bar X)=22-2-17=3$, and it is simply connected.

\begin{lemma} \label{lem:resolution-cyclic}
Consider the quotient singularity $Z=\CC^2/\ZZ_{q}$ given by the action $\e(x,y)=(\e x,\e^{-1}y)$, where $\e=e^{2 \pi \ii/q}$. 
It admits a resolution $\f: \tilde Z \to Z$ so that the exceptional divisor is formed by a chain of $q-1$ 
rational $(-2)$-curves, say 
$\cC=C_1 \cup C_2 \cup \dots \cup C_{q-1}$. The map 
\[
\f_*: \pi_1(\tilde Z - \cC) \to \pi_1(Z - \{0\}) \cong 
\ZZ_{q} \langle \g_0 \rangle
\]
maps the (positive) 
loop $\d_{C_j}$ around each of the curves $C_j$ to 
$\g_0^j$, where $\g_0$ is the loop generating the group $\pi_1((\CC^2 - \{0\}) /\ZZ_{q})$, and $1\leq j \leq q-1$.
\end{lemma}

\begin{proof}
We take the model of the singularity $\CC^2/\ZZ_{q}$ with action $\e(x,y)=(\e x, \e^{-1} y)$, 
$\e=e^{2 \pi \ii /q}$, given by 
$u=x^q$, $w=y^q$, $v=xy$, which is
\[
Z=\{(u,v,w)\ | \ uw=v^{q}\} \subset \CC^3 \, .
\]
We blow up at $(0,0,0)$ to get 
\[
 \tilde Z^1 =\left\{ (a\!:\!b\!:\!c) \times (u,v,w)\, |\, 
u=\l a, v=\l b, w=\l c, ac=\l^{q-2} b^{q} \right\}
\]
The exceptional divisors are
$$
 C=\{\l=0, c=0\}=\{(a\!:\!b\!:\!0) \times 0\}, \quad
   C'=\{\l=0, a=0\}=\{(0\!:\!b\!:\!c)\times 0\}.
$$
These intersect at the point $p=(0\!:\!1\!:\!0) \times 0$. Clearly $C$ and $C'$ are smooth except at $p$; let us see that in fact $p$ is the only singular point of $\tilde Z^1$.

There are three affine charts, the first one $Y$ covers 
$C-\{p\}$ and it is given by setting $a=1$. The coordinates
are then $(u, b,c)$, with $\lambda =u$ 
and the equations of $\tilde Z^1$ are 
$c=u^{q-2}b^q$, which is a smooth surface with
coordinates $(u,b)$.
The curve $C\cap Y$ is defined by $u=0$, 
so a loop around $C$ in this chart is given 
$\d(t)= (u(t), b(t))=(e^{2 \pi \ii t}, 1)$.
Let $\b: \tilde Z^1 \to Z$ be the blow-down map,
$\b(u,b)=(u,v,w)=(u,u b, uc)=(u,ub,u^{q-1}b^q)$. 
The loop $\d$ is mapped to
$$
\b_* \d(t)=(u(t),v(t),w(t))=
(e^{2 \pi \ii t}, 
e^{2 \pi \ii t},  
e^{2 \pi \ii (q-1)t}, ),
$$
which is the generator $\g_0$ of the fundamental group of the link of the singularity, as can 
be noticed since $x^q=u=e^{2\pi\ii t}$ joins $x$ with
$\e x$.

The second affine chart $Y'$ is given by $c=1$, and covers
$C'-\{p\}$. 
The coordinates are $(a,b,w)$, where $\lambda=w$,
and the equations for $\tilde Z^1$ are $a=w^{q-2}b^{q}$, 
which is a smooth surface with coordinates $(b,w)$. The
curve $C'\cap Y'$ is defined by $w=0$,
so a loop around $C'$ given by 
$\d'(t)= (b(t), w(t))=(1,e^{2 \pi \ii t})$
is mapped to
\[
\b_* \d'(t)=(u(t),v(t),w(t))=
(e^{2 \pi \ii (q-1)t}, 
e^{2 \pi \ii t},  
e^{2 \pi \ii t}),
\]
which is the loop $\g_0^{q-1}$. 

Finally, the third affine chart $Y''$ is defined by $b=1$ and
covers the point $p$. 
The coordinates are $(a,v,c)$, where $\lambda=v$,
and the equations for $\tilde Z$ are $ac=v^{q-2}$. This is
a singular surface at the point $p$, completely analogous
to the original $Z$ but with $q-2$ instead of $q$. 
This means that we can repeat the procedure recursively.
Blowing-up at each stage, we introduce two exceptional 
divisors and reduce $q$ in two units. If we get to $q=1$, then
there are no more exceptional divisors, and if we get to
$q=2$, the surface is $uw=v^2$, which is resolved via a single
blow-up obtaining one exceptional divisor.
At the end of the day we get a chain
of $q-1$ rational curves $C_1\cup C_2\cup\ldots \cup C_{q-1}$.
The first blow-up creates $C_1$ and $C_{q-1}$. The second blow-up creates $C_2,C_{q-2}$, and so on.

In coordinates in the third affine chart (the one covering the singularity of $\tilde Z^1$),
the blow-down map is $\beta(a,v,c)=(av,v,cv)$. 
If we denote $\beta_j: \tilde Z^j\to Z$ the
blow-down map from the $j$-th blow-up, it is given
in these affine charts as $\beta_j(a_j,v,c_j)=
(a_jv^j,v,c_jv^j)$. Now consider a loop $\d_j$ 
around $C_j$, $1\leq j <\frac12(q-1)$. 
Then $\beta_{j,j-1}:\tilde Z^j\to\tilde Z^{j-1}$ maps $\delta_j$ to 
the loop 
$$
(\beta_{j,j-1})_*\delta_j(t)=(e^{2 \pi \ii t}, 
e^{2 \pi \ii t},  
e^{2 \pi \ii (q-(2j-2)-1)t} ),
$$
as we computed before. Hence
$$
(\beta_j)_*\delta_j(t)=
(\beta_{j-1})_*(\beta_{j,j-1})_*
\delta_j(t)=
(\beta_{j-1})_*
(e^{2 \pi \ii t}, 
e^{2 \pi \ii t},  
e^{2 \pi \ii (q-2j+1)t})
=
(e^{2 \pi \ii jt}, 
e^{2 \pi \ii t},  
e^{2 \pi \ii (q-j)t}),.
$$
This is the loop $\g_0^j$. We can do 
a similar computation for $\d_j$, $\frac12(q-1)<j
\leq q-1$.

In the case of $q=2d$. There is a final
blow-down map $\beta_{d,d-1}:\tilde Z^d \to \tilde Z^{d-1}$, where $\tilde Z^{d-1}$ has equation
$uw=v^2$. The extra divisor is $C_d=C_{q/2}$, which
is the central divisor in the chain. The image
of the loop $\delta_{d}(t)$ around $C_d$ is the
generator of $\pi_1(\tilde Z^{d-1}-\{0\})=\ZZ_2$,
which has a representative given as
$(\b_{d,d-1})_* \d_d(t)=(e^{2\pi \ii t},e^{2\pi \ii t},e^{2\pi \ii t})$.
Now we compute as above: \[
(\b_d)_* \d_d(t)=(\b_{d-1})_* (\b_{d,d-1})_* \d_d(t)=(e^{2 \pi \ii dt}, e^{2 \pi \ii t}, e^{2 \pi \ii dt})
\]
concluding the proof.
\end{proof}

To proceed, it is convenient to recall a couple of results from \cite{MR-RSME} about construction of symplectic surfaces realizing a given homology class. We start with the concept of nice intersections, that shall be used often in the sequel.

\begin{definition} \label{def:nice}
Let $Z$ be a symplectic cyclic $4$-orbifold. We say that a finite family $S_i$ 
of symplectic surfaces in $Z$ is nice, or that the $S_i$ intersect nicely, if they intersect transversely and positively, no three of them intersect at the same point, and in addition around each point of intersection of $S_i \cap S_j$ there exists a Darboux orbifold chart $(\CC^2,\G,\CC^2/\G, \o_0)$ with coordinates $(z,w)\in \CC^2$, $\G < \ZZ_n$ a cyclic group acting by rotations as $\g(z,w)=(\xi^a_n z, \xi_n^b w)$, $\xi_n=e^{2 \pi \ii/n}$, so that $S_i=\{w=0\}$, $S_j=\{z=0\}$.
\end{definition}

\begin{remark}
A couple of observations on the previous definition:
    \begin{itemize}
        \item If there is only one surface $S_1$, we say that $S_1$ is nice if $S_1=\{w=0\}$ in a Darboux chart as before.
        \item The case where $\G$ is trivial corresponds to a smooth point of the orbifold $Z$ lying in the intersection $S_i \cap S_j$. In this case, one can alternatively take a Darboux chart so that $S_i=\{z+w=0\}$, $S_j=\{z-w=0\}$ via a linear change of coordinates in $\mathrm{SU}(2)$.
        \item The case where $\G$ is non-trivial corresponds to a possibly singular point in the orbifold $Z$, lying in the intersection of two isotropy surfaces $S_i, S_j$.
        \item The surfaces in the codimension-2 isotropy strata of a symplectic cyclic $4$-orbifold are automatically a nice family. This is because the local groups can be assumed to act by rotations, hence the eigenspaces of the action are orthogonal.
    \end{itemize}
\end{remark}

By \cite[Lemma 6]{MRT} and \cite[Proposition 3.3]{MR-RSME}, any family of symplectic surfaces intersecting transversally and positively pairwise can be slightly perturbed so as to intersect nicely. We need a couple of results from \cite{MR-RSME} about the issue of realizing a given homology class by a symplectic surface.

\begin{proposition}\cite[Proposition 5]{MR-RSME} \label{prop:Cn}
    Let $X$ be a symplectic $4$-manifold and $C \subset X$ a symplectic surface with $C^2 > 0$.
    Then for any positive integer $n$ there exists a symplectic surface $C_n$ whose homology class is $[C_n]=n[C]$. The construction can be done inside an arbitrary tubular
    neighborhood of $C$.
    Moreover, for $N\geq 1$, 
    the family of surfaces $C_n$, $1\leq n \leq N$, can be constructed so that
    they intersect nicely.
\end{proposition}

\begin{theorem} \cite[Theorem 4]{MR-RSME} \label{thm:divisors}
Let $X$ be a symplectic manifold, and consider $C_i \subset X$ symplectic surfaces intersecting transversely and positively, and a
homology class 
$[D]=\sum_i a_i [C_i]$ with $a_i>0$ integers, and such that $[D] \cdot [C_j] \ge 0$, for all $j$. 
Then, inside a  
tubular neighborhood of $\bigcup_i C_i$, there exists a symplectic surface $\S$ representing the class $[D]$, and intersecting nicely with all $C_j$. 
\end{theorem}

We can use these results to reprove and improve
some results of \cite{Mu-jems}.

\begin{proposition}\label{prop:theTn}
  There is a collection of smooth symplectic surfaces $T_n$, $n\geq 1$, in a neigbourhood of $T\cup D$,  of genus $g_n=9n^2+1$,
  not intersecting $D$, and such that $[T_n]=n[T_1]$ and all the $T_n$ 
  intersect pairwise nicely.
  \end{proposition}
  
  \begin{proof}
  This is proved in \cite{Mu-jems}, but we  
  give an alternative proof using the techniques from \cite{MR-RSME}.
Let $K$ be the canonical class of the symplectic form. 
Note that $K\cdot T=0$, $K\cdot D=0$ by adjunction.
We start constructing a curve $T_1$ with 
$$
[T_1] = [2T+3D].
$$ 
This is done using Theorem \ref{thm:divisors}.  
Now, using Proposition \ref{prop:Cn}, we obtain $T_n$ with $[T_n]=n [T_1$, and all intersect nicely. Then
$T_n^2=18n^2$ and the genus $g_n=9n^2+1$ satisfies $2g_n-2=18n^2$, since $K\cdot T_n=0$. 
 \end{proof}
  
 \begin{proposition}\label{prop:theA}
 Let $F$ be a fiber of the fibration that intersects the chain $\cC$ transversally at a point of $E_1$. Consider
the configuration of symplectic surfaces $\cC\cup F$. Then there 
is a symplectic surface $A$ of genus $g_A=10$, in a neighbourhood of $\cC \cup F$, not intersecting the chain. 
Also, there is a collection of smooth symplectic surfaces $A_k$, $k\geq 1$, in a neighbourhood of $A$, representing the class $k[A]$ and of genus $g_k=9k^2+1$,
not intersecting the chain $\cC$, and such that all the $A_k$ 
intersect pairwise nicely.
\end{proposition}

\begin{proof}
The construction of $A$  appears in \cite[Proposition 15]{Mu-jems} by a delicate and
very specific analysis of the configuration of curves of the chain. Now it can be done straightforwardly 
applying Theorem \ref{thm:divisors} to 
$\cC \cup F$. Take the homology class of $A$ given by
\[
[A]=[2F+9 E_1 + 8(C_1+C'_1) + 7(C_2+C'_2)+\dots + 2(C_7+C'_7)+ C_8+C'_8]
\]
This satisfies the properties of Theorem \ref{thm:divisors}.
Once $A$ is constructed, the $A_k$ are obtained by Proposition \ref{prop:Cn}.
\end{proof}

 Consider the collection of symplectic surfaces $T_n$, $1\leq n\leq N_1$, and
$T'_m$, $1\leq m\leq N_2$, and $A_k$, $1\leq k\leq K_0$, all of them inside $X$ and disjoint from the $(-2)$-spheres $D, D'$ and the chain $\cC$. The map 
$\pi: X \to \bar X$ contracting $D, D', \cC$ to the singularities $p,p',q$ does not change these surfaces, so we may see them inside $\bar X$ and denote by the same names $T_n=\pi(T_n)$, $T'_m=\pi(T'_m)$ and $A_k=\pi(A_k)$. 
None of the surfaces pass through singular points, and all their intersections are nice by construction. This allows to
assign isotropy coefficients to all of them and make $\bar X$ into a cyclic symplectic $4$-orbifold $X'$ with $T_n, T'_m, A_k$ as isotropy surfaces
(see \cite[Proposition 7]{MRT} for the
technicalities of this assertion).

We take the coefficients of the isotropy surfaces as follows. 
For each $1\leq n \leq N_1$, $1\leq m \leq N_2$,
$1\leq k\leq K_0$, take a prime $p_{nmk}$. The collection of chosen primes should be different, and they 
must satisfy $p_{nmk}> \max(3, n, m,k)$. We assign the following multiplicities:
 \begin{equation} \label{eq:orbifold-X'}
 \begin{aligned}
  m_{T_n} &= \prod_{m,k=1}^{N_2,K_0} p_{nmk} \, , \quad  1 \le n \le N_1\\
  m_{T_m'} &= \prod_{n,k=1}^{N_1,K_0} p_{nmk}^2\, , \quad 1 \le m \le N_2 \\ 
  m_{A_k} &= \prod_{n,m=1}^{N_1,N_2} p_{nmk}^3\, , \quad 1 \le k \le K_0 \, .
  \end{aligned}
 \end{equation} 
Call $X'$ the new orbifold thus constructed from $\bar X$.
Note that for $T_n,T_l$, $n\neq l$, which are intersecting surfaces, we have that
$\gcd(m_{T_n},m_{T_s})=1$. Analogously, 
$\gcd(m_{T_m'},m_{T_l'})=1$ for $m\neq l$, and also $\gcd(m_{A_k},m_{A_l})=1$, for $k\neq l$. 
On the other hand note that $p_{nmk}$ divides $m_{T_n}$, $m_{T_m'}$ and $m_{A_k}$, for
any $n,m,k$, thus
\[
\gcd(m_{T_n},m_{T_m'}, m_{T_{A_k}})\neq 1,
\]
which is in accordance with the fact that the involved surfaces are disjoint.

The next step is to construct a $5$-manifold $M$ as the total space of a Seifert circle bundle over
a cyclic orbifold $X$. For background on this
construction, see \cite{MRT, Mu}. These bundles are classified by a Chern class $c_1(M\to X')$, and
some local invariants associated to the isotropy locus and the singular points. The fact that none of the isotropy surfaces of $X'$ pass through a singular point ensures that we can construct a Seifert bundle $M \to X'$ with a suitable Chern class. The 
homology of $M$ is determined by the following result:

\begin{proposition}[{\cite[Theorem 36]{Mu}}] \label{thm:16MRT}
Suppose that $\pi:M\to X$ is a quasi-regular Seifert bundle with isotropy surfaces $D_i$ with multiplicities $m_i$,
and singular locus $P\subset X$. 
Let $\mu=\lcm (m_i)$. Then $H_1(M,\ZZ)=0$ if and only if
 \begin{enumerate}
 \item $H_1(X,\ZZ)=0$,
 \item $H^2(X,\ZZ)\to \mathop{\oplus}_i H^2(D_i,\ZZ_{m_i})$ is surjective,
 \item $c_1(M/\ZZ_\mu)\in H^2(X-P,\ZZ)$ is a primitive class.
 \end{enumerate}
 Moreover, $H_2(M,\ZZ)=\ZZ^k\oplus (\mathop{\oplus}_i \ZZ_{m_i}^{2g_i})$, $g_i=g(D_i)$ the genus of $D_i$, $k+1=b_2(X)$.
\end{proposition}

We have the following result.

\begin{theorem} \label{thm:K-contact} 
 For any $N_1, N_2, K_0 \geq 1$, there is a Seifert bundle $\pi:M\to X'$ such that $H_1(M,\ZZ)=0$ and 
 \begin{equation} \label{eqn:H2MZ}
  H_2(M,\ZZ)=\ZZ^2 \oplus \bigoplus_{n,m,k=1}^{N_1, N_2, K_0}  \left(\ZZ_{p_{nmk}}^{18n^2+2}
  \oplus \ZZ_{p_{nmk}^2}^{18m^2+2} \oplus \ZZ_{p_{nmk}^3}^{18k^2+2}\right).
  \end{equation}
Moreover, $M$ is spin and admits a $K$-contact structure.
 \end{theorem}

\begin{proof}
We need to check the conditions of Proposition \ref{thm:16MRT}, i.e. that:
\begin{enumerate}
 \item $H_1(X',\ZZ)=0$;
 \item the natural map 
 \[
 H^2(X',\ZZ)\to \mathop{\bigoplus}_{n,m,k} H^2(T_n,\ZZ_{m_{T_n}}) \oplus H^2(T'_m,\ZZ_{m_{T_m}}) \oplus H^2(A_k,\ZZ_{m_{A_k}})
 \]
 is surjective;
 \item the Chern class $c_1(M/\ZZ_\mu)\in H^2(X'-P,\ZZ)$ is primitive, with $P=\{p,p',q\}$ the singularities.
 \end{enumerate}

First, 
clearly $H_1(\bar X,\ZZ)=0$ because $\pi_1(\bar X)=1$. Note that $\bar X \cong X'$ as topological spaces. 
For the second item, 
fix $p=p_{nmk}$ a prime, and look at the map
  \begin{equation*}
 \varpi: H^2(\bar X,\ZZ)\to H^2(T_n,\ZZ_p)\oplus H^2(T_m',\ZZ_{p^2})\oplus H^2(A_k, \ZZ_{p^3}).
  \end{equation*}
We must see that $\varpi$ is surjective.
Note that $T_1,T_1',A\in H_2(\bar X-P,\ZZ) \cong H^2(\bar X,\ZZ)$.
They are mapped to $\varpi(T_1) =(n T_1^2,0,0)=(18n,0,0)$, $\varpi(T_1')=(0,m(T'_1)^2,0)=(0,18m,0)$,
$\varpi(A)=(0,0,kA^2)=(0,0,18k)$. 
Now we use that $n,m,k$ are coprime with $p$ (since $p>n,m,k$) and $p\geq 5$ so that
$\gcd(p,18)=1$; this yields surjectivity.

To proceed, we need to choose $c_1(M)\in H^2(\bar X,\QQ)$ so that it is a symplectic class, and also
$c_1(M/\ZZ_\mu)\in H^2(\bar X-P,\ZZ)$ is primitive. This follows from \cite{MRT} if
we can ensure that 
  \begin{equation*}
  x=\mu \left(\sum_{n=1}^{N_1} \frac{b_n}{m_{T_n}}[T_n] +\sum_{m=1}^{N_2} \frac{b'_m}{m_{T_m'}}[T'_m] + \sum_{k=1}^{K_0} \frac{\b_k}{m_{A_k}} [A_k]\right) \in H^2(\bar X-P,\ZZ)
  \end{equation*}
is primitive, where $b_n,b_m',\b_k$, are the correponding $b_i$ associated to the local invariants, and $\mu=\prod\limits_{n,m,k} p_{nmk}^3$ is the l.c.m. of all the isotropy coefficients. Recall that $[T_n]=n[T_1]$, $[T'_m]=m[T'_1]$, $[A_k]=k[A_1]$, hence 
\[
  x=\mu \left( \sum_{n=1}^{N_1} n \frac{b_n}{m_{T_n}} \right)[T_1] + \mu \left( \sum_{m=1}^{N_2} m \frac{b'_m}{m_{T'_m}} \right)[T'_1] + \mu \left(\sum_{k=1}^{K_0} k \frac{\b_k}{m_{A_k}} \right) [A].
\]

Since we take $p_{nmk} > k$, it follows that the numbers 
\begin{equation} \label{eqn:kk}
k \frac{\mu}{m_{A_k}}=k \prod_{l \ne k} \prod_{n,m} p^3_{nml} \quad , \quad 1 \le k \le K_0
    \end{equation}
are coprime. Certainly, let $p$ be a prime
dividing all numbers (\ref{eqn:kk}). If
$p=p_{nmk}$ for some $k$, then it must
be $p|k$ which is not possible since $p>k$.
If $p$ is not in the list of primes $p_{nmk}$,
then $p|k$ for all $k=1,\ldots, K_0$, which
is not possible again.

Therefore there exist integers $\b_k$ so that $\sum\limits_{k=1}^{K_0} \b_k  k \frac{\mu}{m_{A_k}}=1$, i.e. the coefficient of $[A]$ is $1$. 
Cupping with $[A]\in H_2(\bar X-P,\ZZ)$, we obtain $\la x,[A]\ra=[A]^2=18$. So the only possible divisors of
$x$ are $2$ or $3$. 

The coefficient of $T_1$ in $x$ is
  \begin{equation*}
 b_1  \frac{\mu}{m_{T_1}} + \sum_{n\geq 2} n b_n  \frac{\mu }{m_{T_n}}\, .
  \end{equation*}
As $\mu$ is not divisible by $6$, if we choose $b_1=1$ and $b_n$ divisible by $6$ for $n\geq 2$, then this number 
is coprime with $6$. Then $x$ is not divisible by $2$ or $3$, as required.
By \cite[equation (14)]{Gompf}, the second Stiefel-Whitney class of $M$ is 
  \begin{equation*}
  w_2(M)=\pi^*w_2(\bar X-P) + \sum (m_i-1) \pi^{-1}(D_i)\, .
   \end{equation*}
 As all the isotropy coefficients $m_i$ are odd, then $w_2(M)=\pi^*w_2(\bar X-P)$. Note that the canonical divisor $K_{X'}$ satisfies $K_{X'}\cdot T_1=0$, $K_{X'}\cdot T'_1=0$, $K_{X'}\cdot A=0$,
 hence $K_{X'}=0$. Therefore $w_2(\bar X-P)=0$ and so $w_2(M)=0$ and $M$ is spin.
\end{proof}

\section{Fundamental group of $X'$ and $M$.}

In this section we compute the fundamental group of $M$. For this, we need to study the orbifold fundamental group of $X'$. 

First, it will be useful to introduce the concept of a meridian around a surface in a 4-manifold. Given a $4$-manifold $Z$, a point $p_0 \in Z$, and a surface $S \subset Z$, we say that a loop in $\pi_1(Z - S, p_0)$ is a meridian around $S$ if it is of the form $\a \, \d_{S} \, \a^{-1}$, where $\a=\a_{p_0,q}$ is a choice of path from the base point $p_0$ to a point $q \in \bd \nu_S$ in the boundary of a disc normal bundle of $S$, and $\d_S$ is a loop going around the boundary of the fiber $D^2_{q}$ over the point $q$ in $\nu_S$. It is easy to show that different choices of the point $q$ and the path $\a$ give rise to the same conjugacy class in $\pi_1(Z - S, p_0)$, so when we speak of a meridian we mean any choice of representative of its conjugacy class.

Now, recall that the symplectic manifold $X$ is simply connected and that we contract the surfaces $D,D'$ and the chain $\cC$ to get the symplectic orbifold $\bar X$.
  The singular points of the orbifold $\bar X$ are $P=\{p,p',q\}$. 
  We claim that the fundamental group of 
    \begin{equation*}
   \bar X^0=\bar X-P \cong X-(D\cup D'\cup \cC)
   \end{equation*}
 is generated by loops around the singular points, that is $a,a'$ around $p,p'$, respectively, and
 $b$ around $q$. Note that $a^2=1,a'^2=1, b^{18}=1$.
  This follows by an easy argument using Seifert-Van Kampen 
  with $\bar X= U \cup V$, with $U$ formed by neighborhoods of the singularities $p, p', q$ joined by engrossed paths, $V=X^0$, and $U \cap V$ homotopically equivalent to $S^3/\ZZ_2 \vee  S^3/\ZZ_2 \vee S^3/\ZZ_{18}$.

A presentation of the orbifold fundamental group $\pi_1^{\orb}(X')$ can be obtained as follows. Substract from $\bar X^0$ the codimension 2 isotropy-strata $T_n, T'_m, A_k$, 
then add a loop around the fiber of the disc normal bundle of each stratum (i.e. a meridian around each isotropy surface), and set the \textit{orbifold relations} so that these loops have order equal to 
the isotropy coefficient. In other words:
\begin{equation} \label{eq:meridians}
\pi_1^{\orb}(X')=\frac{\pi_1(\bar X^0 - (\cup_n T_n \cup_m T'_k \cup_k A_k))}{\langle \d_{T_n}^{m_{T_n}}=1,\d_{T'_m}^{m_{T'_m}}=1, \d_{A_k}^{m_{A_k}}=1 \rangle}
\end{equation}
where $\d_{\S}$ denotes a meridian around the surface $\S$, and we quotient by the subgroup normally generated by suitable powers of the meridians to the isotropy surfaces. Let us denote
\[
\bar X^0 - (\cup_n T_n \cup_m T'_m \cup_k A_k) \cong X - (D \cup D'  
\cup \cC \cup_n T_n \cup_m T'_m \cup_k A_k) = X^*
\]
and note that its fundamental group is generated by the loops $a, a', b$ around the singular points $p,p',q$ (equivalently, loops around $D, D'$ and the central component of $\cC$), plus the loops $\d_{T_n}, \d_{T'_m}, \d_{A_k}$ around the codimension $2$ isotropy locus. We must study the homotopy classes of the loops $a, a', b, \d_{T_n},\d_{T'_m}, \d_{A_k}$ in $X^*$, together with the orbifold extra relations.

\medskip

Before adressing the orbifold fundamental group of $X'$, we prove a couple of preliminary lemmas.

\begin{lemma} \label{lem:trivialization}
Let $Z$ be a $4$-manifold, $\S \subset Z$ an embedded surface of positive genus with self-intersection $\S^2=m \in \ZZ$. Consider $\g$ a meridian around $\S$ (the boundary of a fiber some disc normal bundle $\nu_{\S}$). 
Denote $\S^*=\S-\{p\}$, with $p \in \S$ a point.
Let $a_j, b_j$ be generators of $\pi_1(\S^*)$ and let
\[
i: \S^* \x S^1 \xrightarrow{\f} \bd \nu_\S|_{\S^*} \to Z-\S
\]
be the inclusion map, for a choice of trivialisation $\f$ of $\bd \nu_\S|_{\S^*}$.
Then 
\[
\g^m= i_{*}\Big(\prod\limits_{j=1}^{2g} [a_j,b_j]\Big)
= \prod\limits_{j=1}^{2g} [i_{*} a_j,i_{*} b_j] \quad \text{in } \pi_1(\bd \nu_\S) \, .
\]
Note that $\prod\limits_{j=1}^{2g} [a_j,b_j]$ can be identified with a small loop going around the puncture in $\S^*$. 
\end{lemma}

\begin{proof}
The self intersection $\S^2=m$ is the degree of the circle normal bundle $\partial \nu_{\S}$, hence we can take trivializations of $\partial \nu_\S$ of the form $D(2 ) \x S^1$ and $(\S - D(1/2)) \x S^1$,
where $D(r)$ denotes a disc of radius $r>0$ around
$p$. The change of charts from $D(2 ) \x S^1$ to $(\S - D(1/2)) \x S^1$ is given by
\[
\psi:  (z, e^{\ii \theta}) \mapsto (z, z^m e^{\ii \theta}).
\]
We can view $\g^m$ as the loop $\g^m(\theta)=(1, e^{\ii m\theta})$ in $(D(2) - D(1/2)) \x S^1$. 
On the other hand, the image by $\f$ of the loop $\a_0(\theta)=( e^{\ii \theta},1) \subset D(2) \x S^1$ is 
\[
\b_0(\theta)=( e^{\ii \theta}, e^{\ii m \theta}) \subset (\S - D(1/2)) \x S^1 \, .
\]
Since $\a_0$ is the trivial loop in $\pi_1(\partial \nu_\S)$, $\b_0$ is also trivial, hence $\g^m=\g^m \b_0^{-1}=(e^{-\ii \theta},1)$ in $\pi_1(\partial \nu_{\S})$. 
Consider the map 
\[
i_*: \pi_1((\S - D(1/2)) \x S^1) \to \pi_1(\partial \nu_\S).
\]
In $\pi_1((\S - D(1/2)) \x S^1)$, the loop $(e^{-\ii \theta},1)$ is the conmutator of the generators of $\pi_1(\S)$, i.e. $\Pi_j [a_j,b_j]$.
Hence $\g^m$ is homotopic to $i_*( \Pi_j [a_j,b_j])$ in $\pi_1(Z - \S)$.
\end{proof}

We need a preliminary result in order to 
prove abelianity $\pi_1^{\orb}(X')$. 

\begin{theorem} \label{th:meridians-commute}
Let $X$ be a symplectic $4$-manifold and $S=S_1 \subset X$ a symplectic surface with $\ell=S^2>0$. Let $\cU$ be a small tubular neighborhood of $S$. 
There exist symplectic surfaces $S_n \subset \cU$ representing the class $n[S]$, for $1 \le n \le N$,
intersecting nicely, and such that the normal subgroup of $\pi_1(\cU - \cup_n S_n)$ generated by the meridians around the surfaces $S_n$ is abelian. 
In other words, if $c_n$, $c_k$ are meridians around the surfaces $S_n$, $S_k$, then $[c_n,c_k]=1$ in $\pi_1(\cU - \cup_n S_n)$.
\end{theorem}

\begin{proof}
The existence of the surfaces $S_n$ is given by Proposition \ref{prop:Cn},
that is \cite[Proposition 5]{MR-RSME}. It remains to see that the construction is made so that the the meridians around the surfaces $S_n$ commute. Let us briefly review the construction.
    
    First, by \cite[Theorem 3]{MR-RSME}, $\cU$ is symplectomorphic to a neighborhood of the cero section of a holomorphic line bundle $L \to S$ of degree $\ell$ equipped with a suitable Kähler structure (which is linear in the fibers). Hence we can assume directly that $\cU \subset L$ and make all the construction in $L$. For instance we can take $L=\cO_S(p_1+\dots +p_{\ell})$, for some choice of points $p_i \in S$. In $L$, we take $K={N+1 \choose 2}$ holomorphic sections $s_j$ which are $C^0$-small, and whose graphs $\G_j$ intersect pairwise transversely, but possibly with multiple intersections at the points $p_i$. We make a slight $C^0$-small perturbation on $s_j$ so that their graphs remain symplectic surfaces, only two of their graphs intersect at a time, and they intersect nicely, see \cite[Proposition 2]{MR-RSME}. This process can be made so that all the points of intersection of the perturbed graphs $\G_j$ are in the same trivialization $L|_B \cong B \times \CC$, with $B \subset S$ a small ball, because the points $p_1,\dots , p_{\ell}$ can be chosen arbitrarily in $S$ and the perturbations of the graphs are $C^0$-small.

    We group the graphs 
    $\G_j$, $1\le j \le K$, with respect to the partition $K=1+2+3+\dots + N$. For each group of $n$ graphs, say $\G_{j} \cup \dots \cup \G_{j+n}$, we make a symplectic resolution of nice intersections as in \cite[Lemma 3.5]{MR-RSME}, and we obtain the smooth symplectic surface $S_n$ representing the class $n[S]$ in homology. In the local trivialization $L|_B \cong B \times \CC$ with coordinates $(x,y)$, the resolution of nice intersections changes the equation of $\G_{j_1} \cup \G_{j_2}$ near each point of intersection $(a,b) \in \G_{j_1} \cap \G_{j_2}$, from $(x-a)^2-(y-b)^2=0$ to $(x-a)^2-(y-b)^2=\eps$, for small $\eps \in \CC^*$, obtaining thus a smooth surface.

Consider the map 
\[
f:\cU - \cup_n S_n \to S \, , \quad (x,y) \mapsto x
\]
coming from the bundle projection. A generic vertical line intersects the surface $S_n$ at $n$ points since the homology class $[S_n]=n[S]$, so the generic fiber of $f$ is $\CC$ minus $K$ points; however $f$ is not a fibration, as there are fibers with less punctures. Let $\cL$ be the union of those fibers, i.e. fibers which are either tangent to some of the surfaces $S_n$, or lying over some point of intersection in $S_n \cap S_k$. Denote $\pi(\cL)=\{x_i\} \cup \{z_j\}$, with $x_i$ the points associated to tangencies and $z_j$ the points associated to an intersection point of two surfaces. Now, the restriction
\[
f|= \pi:\cU - \cup_{n=1}^N S_n \cup \cL \longrightarrow S - \pi(\cL)
\]
is a locally trivial fibration with fiber $F$ the complex plane $\CC$ minus $K$ points, and base $S - \pi(\cL)$, where $S$ is the zero section. The fundamental group of the total space of a locally trivial fibration is computed by the monodromy theorem. We have an isomorphism 
\[
\pi_1(\cU - \cup_n S_n \cup \cL) \cong \pi_1(F) \rtimes \pi_1(S - \pi(\cL)) \, ,
\]
where $\pi_1(S - \pi(\cL))$ acts on $\pi_1(F)$ by monodromy. In the semi-direct product we denote the operation as $(g,h) \cdot (g',h')=(g(g')^{h}, hh')$, where the super-index denotes the monodromy action. Let us call $g_1, \dots , g_K$ the generators of $\pi_1(F)$, and $\g_i, \d_j$ the loops of $\pi_1(S - \pi(\cL))$ going around the points $x_i, z_j$,
respectively. On the other hand, by a standard transversality argument, we have an epimorphism
 \begin{equation} \label{eq:surjective-map}
 \pi_1(\cU - \cup_n S_n \cup \cL) \longrightarrow \pi_1(\cU - \cup_n S_n),
 \end{equation}
whose kernel contains the loops $\g_i, \d_j$. Note that in the semi-direct product we have 
\[
(g,1) \cdot (1,\g_i)=(g^{\g_i}, \g_i),  \quad (g,1) \cdot (1,\d_j)=(g^{\d_j}, \d_j) \, .
\]

Since $\g_i, \d_j$ are in the kernel of the epimorphism \eqref{eq:surjective-map}, this induces relations in $G=\pi_1(\cU - \cup_n S_n)$, called the monodromy relations, so that $G$ is generated by
\begin{itemize}
    \item $g_1, \dots , g_{K}$, the image of the loops in $\pi_1(F)$;
    \item the image of the loops in $\pi_1(S)$;
    \item and the relations $g=g^{\g_i}$, $g=g^{\d_j}$ hold in $G$, for any $g \in \pi_1(F)$.
\end{itemize}

Note that the subgroup $H=\la g_1, \dots , g_{K} \ra$ is normal in $G$, since it is the image of the normal subgroup $\pi_1(F)$ of the semi-direct product by the epimorphism \eqref{eq:surjective-map}. It remains to see that $H$ is abelian, i.e. that the monodromy relations $g=g^{\g_i}$, $g=g^{\d_j}$ impose abelianity in $H$. To see this, we choose a base point $p_0 \in S - \pi(\cL)$ inside the trivialization $L|_B$. The fiber $F_{p_0}$ is the plane $\CC$ minus $K$ points $y_1,\dots , y_K$, and the loops $g_1, \dots g_K$ are loops in $F_{p_0}$ based at $p_0$ and going around these points. We order these loops cyclically.
Take the loop $g_1$ around the point $y_1$, which is a meridian around some of the surfaces, say $S_{i_1}$. Consider the following loop $g_2$ (with respect to the cyclic order); we have two cases:

\begin{itemize}
\item \textbf{Case 1:} Suppose that $g_2$ is also a meridian around $S_{i_1}$. Recall that $S_{i_1}$ was constructed by resolving the intersection points of some graphs $\G_1, \dots \G_{i_1}$. Hence, we can identify $g_1$, $g_2$ as meridians around graphs, say $\G_1$ and $\G_2$. The graphs $\G_1$ and $\G_2$ initially had $\ell>0$ points of transversal intersection that were later resolved. The resolution changes a local equation of type $(x-a)^2-(y-b)^2=0$ near the point of intersection $(a,b)$ to another of the form $(x-a)^2-(y-b)^2=\eps$. This eliminates the singularity but creates two ordinary points of vertical tangent, the points $(x_i^{\pm}, b)$, with $x_i^{\pm}=a \pm \sqrt{\eps}$. Hence, we can find a path $\tilde \g$ contained in $S_{i_1}$, so that it starts at $y_1 \in F_{p_0}$, travels along the graph $\G_1$, arrives to a point of vertical tangency $(x_i,b)$ that appears when resolving $\G_1 \cap \G_2$, and then goes through the graph $\G_2$ to the point $y_2 \in F_{p_0}$, with the same projection on the base. Projecting to the base we get the loop $\g=\a \, \g_i \, \a^{-1}$ that goes from $p_0$ to a point near the tangency $x_i$ via the path $\a$, then goes around $x_i$ via $\g_i$, and then comes back to $p_0$ via $\a^{-1}$. The monodromy induced by this loop $\g$, being an ordinary tangency, is well known; it only affects the meridians $g_1, g_2$, via $g_1^{\g}=g_2$, $g_2^{\g}=g_2^{-1} \, g_1 \, g_2$. Hence, the monodromy relations impose that $g_1=g_2$ in the group $G$. In particular, $g_1$ and $g_2$ commute.

\item \textbf{Case 2:} Suppose that $g_2$ is a meridian around another surface $S_{i_2}$. The surfaces $S_{i_1}$ and $S_{i_2}$ intersect at many points, so we may take one of them, say $(z_j, b)$. In a similar way as in the previous item, we take a loop $\tilde \g$ that starts at $y_1$, goes to $z_j$ through $S_{i_1}$, then comes back to $y_2$ through $S_{i_2}$; projecting we get $\g= \a \, \d_j \, \a^{-1}$, with $\d_j$ a small loop around $z_j$. The monodromy of $\g$ is that of an ordinary node, so we have $g_1^{\g}=g_2^{-1} \, g_1 \, g_2$, $g_2^{\g}=(g_1 g_2)^{-1} \, g_2  \, (g_1g_2)$. Hence, the relations $g_1=g_1^{\g}$ and $g_2=g_2^{\g}$ in the group $G$ imply that $g_1$ and $g_2$ commute.
\end{itemize}

We have proved so far that $g_1$ commutes with $g_2$, let us see now that it also commutes with $g_3$. We could try a similar argument as before; now, however, the loop $g_2$ lies in between $g_1, g_3$, and this affects the monodromy. This is fixed changing $g_3$ by $g'_3=g_2 \, g_3 \, g_2^{-1}$, so now $g_1$ and $g'_3$ are consecutive loops with respect to the cyclic order in the fiber $F_{p_0}$. Now we are in the same situation as above; we have two cases: either $g_1$ and $g_3$ are meridians around the same surface $S_{i_1}$, or $g_3$ is a meridian around another surface $S_{i_3}$. In either case, we can proceed as before, finding a loop in the base around a point of tangency in the first case, or a loop around a point of intersection of $S_{i_1}$ and $S_{i_3}$ in the second case. In any case it follows that $g_1$ and $g'_3$ commute. But we know already that $g_1$ and $g_2$ commute, so $g_1$ also commutes with $g_3=g_2^{-1} \, g'_3 \, g_2$.

An easy induction process shows that $g_1$ commutes with all loops $g_2 , \dots , g_K$. Of course, the loop $g_1$ around the surface $S_{i_1}$ is in no way special, so in fact this argument applies to show that $g_j$ commutes with $g_l$ for any $1 \le j, l \le K$. This concludes the proof.
\end{proof}

From the above result we deduce the following:

\begin{corollary} \label{cor:meridians-commute}
With the notations from \eqref{eq:meridians}, the surfaces $T_n$ satisfy that any two of the meridians $\d_{T_n}$, $1\leq n\leq N_1$, commute in $\pi_1(X^*)$. The same is true for the surfaces $T'_m$ and $A_k$: any two of the $\d_{T'_m}$, $1 \le m \le N_2$, and any two of the $\d_{A_k}$, $1 \le k \le K_0$ commute in $\pi_1(X^*)$.
\end{corollary}

Recall that in principle we do not know whether $\d_{T_n}$ commutes with $\d_{T'_m}$ or $\d_{A_k}$, nor whether $\d_{T'_m}$ and $\d_{A_k}$ commute. We tackle this issue in the following result, where we prove that $\pi_1^{\orb}(X')$ is finite cyclic and that $M$ is simply connected, thus completing the proof of Theorem \ref{th:main-0}. Note that this fixes a small innacuracy from \cite{Mu-jems}, where it was claimed that $\pi_1^{\orb}(X')$ is trivial; we ignore if this is the case, but it is irrelevant for our purposes, as abelianity of $\pi_1^{\orb}(X')$ is enough to ensure that the $K$-contact manifold $M$ is simply connected.

\begin{theorem} \label{thm:pi1(M)}
Let $X'$ be the orbifold constructed in \eqref{eq:orbifold-X'}.
The orbifold fundamental group $\pi_1^{\orb}(X')$ is finite cyclic. As a consequence, $\pi_1(M)=1$.
\end{theorem}

\begin{proof}
We will see that all the loops $a,a',b,\d_{T_n}, \d_{T'_m}, \d_{A_k}$, except possibly $b$, are trivial in $\pi_1(X^*)$ once the extra relations on  the orbifold fundamental group (\ref{eq:meridians}) are added.

\noindent \textbf{Step 1: relations coming from the fiber.}
First fix a smooth fiber $F_0$ of the elliptic fibration $X \to \CP^1$. Recall that:
\begin{itemize}
\item The vanishing cycles in $F_0$ are $(1,1)$, $(1,-2)$, $(2,-1)$. The cycle $(1,1)$ was used to construct $D$ and $T'$, while the cycle $(1,-2)$ was used to construct $D'$ and $T$. The cycle $(2,-1)$ has not been used for constructing the surfaces.
\medskip
\item The tori $T$ and $T'$ have also horizontal components $S^1(r), S^1(r')$, given by parallel loops going around the fiber substracted when doing the connected sum $X= S_1 \#_{F_1 \cong F_2} S_2$. These project to loops in $\CP^1$.
\end{itemize}
Let $\alpha,\beta \in \pi_1(F_0)$ be the standard generators of the fundamental group of the fiber $F_0$ of $X$. 

\noindent \textbf{Vanishing cycle $(2,-1)$:} we move the fiber $F_0$ so that it lies over a point  in $\CP^1$ away from images of the horizontal loops of $T$ and $T'$. 
In this way, the third vanishing cycle $(2,-1)=\a^2 \b^{-1}$ of $F_0$ contracts without touching a neighborhood of $T \cup D$ and $T' \cup D'$, because the homotopy involves a path in the base going to the third nodal fiber, and this path is disjoint from the paths associated to $D$ and $D'$. The homotopy does not touch a neighborhood of the chain $\cC$ either, because the section $E_1$ does not intersect the nodal curves at the nodes, hence $\a^2 \b^{-1}$ can be contracted in a small neighborhood of the node without touching $E_1$. We conclude that $\alpha^2\beta^{-1}=1$ in $\pi_1(X^*)$, so $\beta=\alpha^2$.

\noindent \textbf{Vanishing cycle $(1,1)$:} we can null-homotop the cycle $\a \b$ via a small perturbation of a hemisphere of $D$, i.e. via a $(-1)$-disc which intersects $D$ transversely (and negatively) at one point. This gives a relation $\a \b \, a=1$ in $\pi_1(X^*)$. Recall that $a=a^{-1}$, so we can ignore the orientation of $a$. This, together with the previous relation $\b=\a^2$, yields $\a \b=\a^3=a$.

\noindent \textbf{Vanishing cycle $(1,-2)$:} with an analogous reasoning we get a homotopy from $\a \b^{-2}$ to the constant loop via a perturbed hemisphere of $D'$. This homotopy intersects $D'$ transversely at one point, so we have $\a \b^{-2} a'=1$ in $\pi_1(X^*)$. Adding the relations from the other two cycles we get 
\begin{equation} \label{eq:relation-vanishing-discs}
a'=\a \b^{-2}=\a^{-3}=a^{-1}=a=\a \b=\a^{3} \, .
\end{equation}

\noindent \textbf{Step 2: loops $\d_{A_k}$ around $A_k$.}
Next, take the surface $A$ which lies in a neighbourhood of $F\cup \cC$, and has genus $10$, and self-intersection
 $A^2=18$. The surface $A$ lies in a neighbourhood of $F\cup \cC$, and has genus $10$, and self-intersection
 $A^2=18$. The surface $A$ has homology class
 \[
 [A] = [2F + 9E_1+ 8(C_1+C'_1)+ \dots + 2(C_7+C'_7)+(C_8+C'_8)] \, .
 \]
 Consider the sections $E_2, E_3$, which are $(-2)$-spheres. The section $E_2$ intersects the curves $C_4, C'_4$ and $F$ in one point, while $E_3$ intersects $C_7, C'_7$ and $F$.
  Therefore we have the intersection products
 \[
 E_2 \cdot A_k=k E_2 \cdot A=12k \ , \qquad E_3 \cdot A_k=k E_3 \cdot A=6k \, .
 \]
 On the other hand, the meridians around $C_4, C'_4$ are $b^5, b^{13}$ respectively, while the meridians around $C_7, C'_7$ are $b^2, b^{16}$, by Lemma \ref{lem:resolution-cyclic}. 
 Since $E_2, E_3$ are spheres disjoint from $D, D', T_n, T'_m$, we have the relations
\[
 b^5 \cdot b^{13} \prod_{j=1}^{K_0} \d_{A_j}^{12j} =1 \ , \qquad
b^2 \cdot b^{16} \prod_{j=1}^{K_0} \d_{A_j}^{6j}=1\ ,
 \]
in $\pi_1(X^*)$. Since $b^{18}=1$, and the meridians $\d_{A_j}$ commute by Corollary \ref{cor:meridians-commute}, we deduce, 
swapping to additive notation,
that $6 \sum_j j \d_{A_j}=0$. Fix $k$, multiply by $\prod_{j \ne k} m_{A_j}$, and use the orbifold relations to obtain
\[
6 k (\prod_{j \ne k} m_{A_j}) \d_{A_k}=0\ , \qquad
m_{A_k} \d_{A_k}=0.
\]
As $6 k \prod_{j \ne k} m_{A_j}$ and $m_{A_k}=\prod_{n,m} p^3_{nmk}$ are coprime, we get 
that $\d_{A_k}=0$.
Recall that this argument is also valid when $K_0=1$, i.e.\ when we have only the surface $A$, as in the paper \cite{Mu-jems}.

 \noindent \textbf{Step 3: loops $a$ around $D$ and $c_n$ around $T_n$.}
Take the surfaces $T_n$ lying  in a neighbourhood of 
$T\cup D$. Let us denote $c_n=\d_{T_n}$ a meridian around $T_n$, and $m_n=m_{T_n}$ its isotropy coefficient, $1 \le n \le N_1$; recall
that $a$ is a small loop around $D$, and $a^2=1$. We know that $[c_n,c_j]=1$ in $\pi_1(X^*)$ for all $n,j$, by Corollary \ref{cor:meridians-commute}. 

Recall that $T=s \times (\a \b^{-2})$, where $s=S^1(r)$ is a horizontal circle around the fiber substracted when doing the fiber connected sum. The intersection product $T \cdot T_n= 9n$, and $T \cdot D=3$. We have that
 \[
 T^*=T \cap X^* = T - ((T \cap D) \cup_n (T \cap T_n))
 \]
 is a torus with $3+ \sum\limits_{n=1}^{N_1} 9n$ punctures, giving a relation 
  \begin{equation} \label{eq:relation}
  a^3 c_1^9c_2^{18}\ldots c_{N_1}^{9N_1}=[\alpha\beta^{-2} , s] \quad \text{ in } \pi_1(X^*) \, .
 \end{equation}
We saw above that $\a \b^{-2}=a$, and $a^2=1$, so we deduce
\[
a c_1^9\ldots c_{N_1}^{9N_1}=a s a s^{-1} \implies s a s^{-1}= c_1^9\ldots c_{N_1}^{9N_1} \, .
\]
The loops $c_1, \dots , c_{N_1}$ commute, and by the orbifold relations each satisfies $c_j^{m_j}=1$ with $m_j=m_{T_j}$ an odd number. If we write $m=\prod_j m_j$, then $sa^m s¨^{-1}=(s a s^{-1})^m=(c_1^9\cdots c_{N_1}^{9N_1})^m=1$, so we get $a^m=1$ for an odd number $m$. Since $a^2=1$, this gives $a=1$ in $\pi_1^{\orb}(X')$. Therefore we have $c_1^9 \cdots c_{N_1}^{9N_1}=1$. Since the $c_j$ commute, we use additive notation, so we have $9 \sum_j j \,c_j=0$. If we fix $n$ and multiply by $\prod_{j \ne n} m_j$ we get
$$
9n (\prod_{j \ne n} m_j)  \, c_n=0\ , \qquad
 m_n \, c_n =0, 
$$
where the second equation is the orbifold relation.
Now, by the choice of the primes $p_{nmk}$ we have that $m_n$ is coprime with $9n \prod_{j \ne n} m_j$, so we deduce that $c_n$ is trivial in $\pi_1^{\orb}(X')$ for all $1 \le n \le N_1$.

\noindent \textbf{Step 4: loops $a'$ around $D'$ and $c'_m$ around $T'_m$.} This time we use $T'=s' \times (\a \b)$, being $s'=S^1(r')$ a parallel copy of $S^1(r)$.
The intersection numbers are as before, so we have a relation 
 \[
  a' c_1'{}^9\ldots c_{N_2}'{}^{9N_2}=[\alpha\beta , s']=[a', s'] \quad \text{ in } \pi_1(X^*) \, .
\]
On the other hand, by \eqref{eq:relation-vanishing-discs} we have $a'=a=1$. Now, as the loops $c'_j$ commute, using additive notation we get the same relation as before: $9 \sum_j j\, c'_j=0$, and it follows by the same argument that all $c'_j$ are trivial in $\pi_1^{\orb}(X')$.

To sum up, we have proved that the generators $a,c_n, a', c'_j, \d_{A_k}$ of $\pi_1^{\orb}(X')$ vanish, except possibly the loop $b$, so $\pi_1^{\orb}(X')$ is generated by $b$ with $b^{18}=1$, hence it is a finite cyclic group of order at most 18.
Once proved that $\pi^{\orb}_1(X)$ is abelian, we get that $\pi_1(M)$ is also abelian using the exact sequence on homotopy groups
\[
\ZZ \to \pi_1(M) \to \pi_1^{\orb}(X) \to 1
\]
coming from the Seifert circle bundle $M \to X$. Finally, using that $H_1(M,\ZZ)=0$ from Theorem \ref{thm:K-contact}, we get that $M$ is simply connected.
\end{proof}

\begin{remark}
It is easy to show that in fact $b^3=1$ in $\pi_1^{\orb}(X')$, so the orbifold fundamental group has order at most three. Indeed, the loop $s=s'$ representing $S^1(r)$ can be homotoped to $b^5$ inside the section $E_2$, and can also be homotoped to $b^2$ inside the section $E_3$, so $s=b^5=b^2$ and $b^3=1$.
\end{remark}

\section{Non-existence of Sasakian structure}

Now we aim to see that the $K$-contact manifolds $M$ constructed in Theorem \ref{th:main-0} do not admit any Sasakian structure. In order to do this, we will study properties of the Kahler cyclic orbifolds that would appear as spaces of leaves of a tentative Sasakian structure, and finally we will derive a 
contradiction. 

If $M$ admits a Sasakian structure, then it admits a quasi-regular one by \cite{Rukimbira}. So $M$ is the
total space of a Seifert bundle 
 \begin{equation} \label{eqn:MtoY}
 \varpi: M\too Y
 \end{equation}
over a cyclic Kähler orbifold $Y$. The
homology of $M$ is given by (\ref{eqn:H2MZ}). Using
Proposition \ref{thm:16MRT}, applied to
(\ref{eqn:MtoY}) we have that $b_1(Y)=0$ and $b_2(Y)=3$.
 The (finite) set of singular points of $Y$ is denoted by $P$. Also, there are complex curves $D_i$
forming the (complex) codimension-1 isotropy strata of $Y$, with some
multiplicity $m_i$. This family of curves is nice according to Definition \ref{def:nice}. Looking at the 
torsion in (\ref{eqn:H2MZ}), we see that for
each prime $p_{nmk}$, there is a triplet of
surfaces 
 \begin{equation} \label{eqn:enmk}
 \e^{nmk}=\left(D_1=D_1^{nmk}, D_2=D_2^{nmk}, D_3=D_3^{nmk}\right) ,
 \end{equation}
whose isotropy coefficients $m_1,m_2,m_3$ satisfy
that $p_{nmk}\parallel m_1$, $p_{nmk}^2\parallel  m_2$, $p_{nmk}^3\parallel m_3$, where the sign 
$\parallel$  means the maximal prime power dividing an integer. The genera are also recovered from the torsion of $M$, and they are given by $g(D_1)=9n^2+1$, $g(D_2)=9m^2+1$, $g(D_3)=9k^2+1$.

We note that it may happen that curve $D_i$
appears in several of the triplets $\e^{nmk}$,
depending on the prime factors of the 
isotropy coefficient $m_i$.
By (2) in Proposition \ref{thm:16MRT}, 
the three curves in (\ref{eqn:enmk}) are
disjoint curves and generate the rational
homology $H_2(Y,\QQ)$. 

We will abbreviate this by saying that $\e=(D_1,D_2,D_3)$ is an \emph{orthogonal basis}, or simply a \emph{basis} if the context is clear.
The existence of such a basis implies that $b^{1,1}(Y)=b_2(Y)=3$. By the Riemann-Hodge bilinear relations for orbifolds, one of these curves has positive self-intersection, and the other two have negative self-intersection. We recall now some terminology and results from \cite{Mu-jems}.

\begin{proposition} \cite{Mu-jems} \label{prop:effective-divisor} 
Let $K_Y$ the canonical divisor of $Y$, and $\e=(D_1,D_2,D_3)$ be an orthogonal basis of positive genera $g_1,g_2,g_3$. The $\QQ$-divisors $K_Y+D_i$, $K_Y+D_1+D_2+D_3$ are effective (and non-zero).
\end{proposition}

From the adjunction formula for symplectic suborbifolds it follows that for any nice curve $D$ (see remark after Definition \ref{def:nice}) we have 
\[
K_Y \cdot D + D^2=-e^{\orb}(D)=2g(D)-2+\sum_p (1-\tfrac{1}{d_p}) \, .
\]
Where $e^{\orb}(D)$ denotes the orbifold Euler characteristic of $D$, which is a $2$-orbifold with conical singularities $p$ of order $d_p$.
If we consider an orthogonal basis $\e=(D_1,D_2, D_3)$, and express $K_Y=\sum \a_i D_i$, then 
\[
\a_i=\frac{K_Y \cdot D_i}{D_i^2}=\frac{-e^{\orb}(D_i)-D_i^2}{D_i^2}=\frac{\chi_i-D_i^2}{D_i^2} \, .
\]
Assume that $D_1$ is the positive curve and denote $D_1^2=m_1$, $D_2^2=-m_2$, $D_3^2=-m_3$, with $m_i>0$. Denote also $\chi_i=-e^{\orb}(D_i)$.

\begin{corollary} \label{cor:m1-le-chi1}
Let $\e=(D_1,D_2,D_3)$ be a basis of positive genera $g_1,g_2,g_3$, and assume that $D_1$ is the positive curve. Then 
\[
D_1^2=m_1<\chi_1=-e^{\orb}(D_1)\, .
\]
\end{corollary}

\begin{proof}
Applying adjunction we get
\[
K_Y+D_2=\frac{\chi_1-m_1}{m_1}D_1-\frac{\chi_2}{m_2}D_2-\frac{\chi_3+m_3}{m_3}D_3\, ,
\]
and this divisor is anti-effective if $m_1 \ge \chi_1$, a contradiction with Proposition \ref{prop:effective-divisor}.
\end{proof}

\begin{definition}[Proj-equivalence]
Two bases $(D_1,D_2,D_3)$, $(D'_1,D'_2,D'_3)$ are proj-equivalent if, up to reordering the indices, there are scalars $\l_i>0$ so that $[D_i]=\l_i [D'_i] \in H_2(Y,\QQ)$.
\end{definition}

\begin{remark}
Suppose the bases $(D_1,D_2,D_3)$, $(D'_1,D'_2,D'_3)$ are proj-equivalent, and assume $D_2$, $D_3$, $D'_2$, $D'_3$ are the negative curves. Then $D_2=D'_2$ and $D_3=D'_3$. Indeed, if we had $D_2 \ne D'_2$, then the intersection $D_2 \cdot D'_2 \ge 0$ since both are irreducible complex curves. But $[D_2]=\l_2 [D'_2]$, so $D_2 \cdot  D'_2=\l_2 D'^2_2<0$, a contradiction.
\end{remark}

\begin{lemma} \label{lem:proj-equiv}
Let $\e=(D_1,D_2,D_3)$, $\e'=(D'_1,D'_2,D'_3)$, $\e''=(D''_1,D''_2,D''_3)$ three bases, and assume that they are ordered so that the first curve is the positive one. The following holds:
\begin{itemize}
    \item If $[D'_1]=\l_1 [D_1]$ for some $\l_1>0$ then $\e$ and $\e'$ are proj-equivalent, so in particular $D_2=D'_2$, $D_3=D'_3$.
    \medskip
    \item If $[D_3]$, $[D'_3]$, $[D''_3]$ are proportional in $H_2(Y,\QQ)$ then two of the bases $\e, \e', \e''$ are proj-equivalent.
\end{itemize}
\end{lemma}
\begin{proof}
See \cite[Lemmas 5.2 and 5.3]{Mu-jems}.
\end{proof}

\begin{definition}[Good basis] \label{def:good}
A basis $\e=(D_1,D_2,D_3)$ is called \emph{good} if none of its curves passes through any singular point of $Y$.
\end{definition}

\begin{remark}
For a good basis we have that $[D_i] \in H_2(Y,\ZZ)$, $D_i^2 \in \ZZ$, and $\chi_i=2g_i-2 \in \ZZ$.
\end{remark}

Denote by $P \subset Y$ the set of singular points of $Y$, which is a finite set. We have the following universal bound for the singular points, proved in \cite[eq.\ (11)]{Mu-jems}.

\begin{proposition}\label{prop:bound-singular}
Suppose that $Y$ has three bases $\e, \e', \e''$ of positive genera $g_i, g'_i, g''_i$, $i=1,2,3$, and with all the curves distinct (i.e.\ in total nine distinct curves). Then 
\[
\# P \le -6 + 4 \sum_{i=1}^3 (g_i+g'_i+g''_i).
\]
\end{proposition}

\section{Explicit bounds for geometric quantities of $Y$} \label{sec:5}

Let $Y$ be the orbifold in (\ref{eqn:MtoY}), that 
is the base of a Seifert bundle for the manifold
$M$ in Theorem \ref{th:main-0} or in
Theorem \ref{thm:main-jems}. Recall that $Y$
is a cyclic orbifold, where the isotropy locus of the orbifold consists of isolated singular points $p \in P$, and some complex curves given
by the triples (\ref{eqn:enmk}) for 
$1\leq n\leq N_1$, $1\leq m\leq N_2$, 
$1\leq k\leq K_0$. We will only use the triples
$$
 \e^{nm}=\e^{nm1}, \quad 1\leq n,m\leq N,
 $$
for $N=\min(N_1,N_2)$. In this way, we handle
the K-contact manifolds from Theorem \ref{th:main-0} and Theorem \ref{thm:main-jems} simultaenously. 
Our objective is to produce a lower bound of $N$,
thus proving Theorem \ref{th:main-2}.

From the discussion above, we know that for any pair $(n,m)$ with $1 \le n,m \le N$ there exists a triple of disjoint curves $\e^{nm}=(D^{nm}_1, D_2^{nm}, D_3^{nm})$ representing a $\QQ$-basis for $H_2(Y,\QQ)$, i.e.\ an orthogonal basis. These curves have genera 
\[
g(D^{nm}_1)=9n^2+1 \, , g(D^{nm}_2)=9m^2+1 \, , g(D^{nm}_3)= 10.
\]
The complex curves $D_i^{nm}$ are codimension-one isotropy curves of $Y$. Since $Y$ is cyclic, these curves intersect nicely, and each singular point $p \in P$ is in at most two of the curves.
We may abbreviate this situation by saying that $Y$ has $N^2$ orthogonal bases.

In this section we will assume that our orbifold $Y$ has such many bases $\e^{nm}$, for $1 \le n,m \le N$, and we will deduce some consequences. Later on, for a suitable choice of $N$ large enough, this will lead to a contradiction, thus to non-existence of such an orbifold $Y$. We shall determine explicitly lower bounds for the value of $N$.

\subsection{Bound for the singular points}

\begin{proposition}\label{prop:five-bases}
Suppose that $Y$ has five bases $\e(i)$ for $1 \le i \le 5$, and the genera of the curves are $g_1(i), g_2(i), 10$, with $g_1(i), g_2(i)>10$ ten distinct numbers. Then at least three of the bases have all their curves distinct. 
In particular, Proposition \ref{prop:bound-singular} applies, and there is a bound for $\# P$.
\end{proposition}

\begin{proof}
Suppose that among the five curves of genus $10$ in the bases $\e(i)$ there are only two distinct. Then three of the bases would have the same genus $10$ curve. By Lemma \ref{lem:proj-equiv} it follows that two of the bases $\e(i)$ are proj-equivalent. But this is not possible, since proj-equivalent bases always share two of the curves, and the choices of genera forbid this.
\end{proof}

\begin{corollary}
There exists a constant $\t_0$ so that if the orbifold $Y$ has $N^2$ orthogonal bases $\e^{nm}$ with $N \ge 11$ then $\# P \le \t_0$.
\end{corollary}

\begin{proof}
We take $n,m$ ranging over the $10$ distinct values in $\{2,\dots,11\}$, and this gives $5$ bases $\e^{nm}$ with genera as in Proposition \ref{prop:five-bases}, hence we can apply Proposition \ref{prop:bound-singular}.
\end{proof}

\noindent \textbf{Computation of $\t_0$:}
We need five bases $\e(i)=(D_1(i), D_2(i), D_3(i))$ with genera $g_1(i)$, $g_2(i)$, $g_3(i)=10$, $1\leq i \leq 5$, so that all $g_1(i)$, $g_2(i)$ are distinct. We take:
\begin{itemize}
    \item $(n,m)=(2,3)$ gives $\e^{23}$ with genera $37,82,10$.
    \item $(n,m)=(4,5)$ gives $\e^{45}$ with genera $145,226,10$.
    \item $(n,m)=(6,7)$ gives $\e^{67}$ with genera $325,442,10$.
    \item $(n,m)=(8,9)$ gives $\e^{89}$ with genera $577,730,10$.
    \item $(n,m)=(10,11)$ gives $\e^{10 \, 11}$ with genera $901,1090,10$.
\end{itemize}
Among these $5$ bases, there are three of them with all the curves different. Then the number of singular points can be bounded by
Proposition \ref{prop:bound-singular} to be:
\begin{align*}
\# P \le & -6 + 4 \cdot (325+ 442+ 577+730 + 901+ 1090 +30) = 16\,374=\t_0 \, .
\end{align*}

\subsection{Existence of a good basis}
We aim to see that, if we assume the existence of sufficiently many bases $\e^{nm}$, then we can ensure that some of them are good (see
Definition \ref{def:good}).
Let us introduce a definition which will help in this process.

\begin{definition}[$\mu_0$-bounded basis]
We say that a basis is \emph{$\mu_0$-bounded} if it is of the form $\e^{nm}=(D_1^{nm},D_2^{nm}, D_3^{nm})$ with $2 \le m \le \mu_0$ and $n > \mu_0$. 
\end{definition}
The genera of the curves of a $\mu_0$-bounded basis are \[
g(D_1^{nm})=9n^2+1 >9 \mu_0^2+1 \ge g(D_2^{nm})=9m^2+1 >g(D_3^{nm})=10.
\]
Two of the curves are negative (i.e.\ have negative self-intersection), and the other one is positive. In principle, we do not know which one is the positive, but later on we will study this.
Note that the curves of such a basis may be not good (i.e. they may contain some singular points of $Y$). 
We aim to see that choosing a suitable bound $\mu_0$ ensures the existence of $\mu_0$-bounded and \emph{good} bases, while allowing us to gain some control over the genus of the curves $D_2$.

We give a method to construct suitable $\mu_0$-bounded basis.
For any $m \in [2,\mu_0]$ we choose some number $n_m \ge \mu_0+1$, so that the function $m \mapsto n_m$ is injective. For each pair $(n_m,m)$ consider the basis $\e(m)=\e^{n_m m}$, $2 \le m \le \mu_0$. We have in total $\mu_0-1$ such bases.

\begin{lemma} \label{lem:three-bases}
Any singular point $p \in P$ is in at most four of the bases $\e(m)$ constructed above.
\end{lemma}

\begin{proof}
Suppose the contrary, and let $\e(m_i)$, with $1 \le i \le 5$ so that $p$ is in some of the curves of each $\e(m_i)$. Since $p$ cannot be in more than two curves, among these five curves containing $p$ there must be repetitions; and, as the genera of $D_1^{n_{m_i} m_i},D_2^{n_{m_i} m_i}$ are all distinct, it follows that among the curves $D_3^{n_{m_i} m_i}$ for $1 \le i \le 5$, there are only two different curves, so the same curve $D_3$ is in three of the bases, say $\e(m_1)$, $\e(m_2)$, $\e(m_3)$. By Lemma \ref{lem:proj-equiv}, at least two of these bases must be proj-equivalent, say $\e(m_1)$ and $\e(m_2)$. But any pair of proj-equivalent bases share two of the curves (the negative ones), and this contradicts the choice of genera of the curves for the bases $\e(m)$.
\end{proof}

Since there are at most $\t_0$ singular points in $P$, we deduce that at most $4 \t_0$ of the bases $\e(m)$ may not be good bases. Now let us choose the numbers $n_m$ the smallest possible, i.e.\ $n_m=\mu_0-1+m$, for $2 \le m \le \mu_0$. Let us define 
\[
N_1=\max \{n_m: m \in [2,\mu_0]\}=2 \mu_0-1 \, .
\]

\begin{corollary}
For any choice of $\mu_0 \ge 4 \t_0+2$, any Kähler orbifold $Y$ with bases $\e^{nm}$, for $2 \le n,m \le N_1=2 \mu_0-1$, has at least a good $\mu_0$-bounded basis.
\end{corollary}

\begin{remark} \label{rem:5.6}
 In \cite{Mu-jems} the numbers $n_m$ are chosen to be prime, and this yields a number considerably bigger than the number $N_1=2 \mu_0-1=8 \t_0+3=130 \, 995$ given above. The choice of $n_m$ as prime numbers is unnecessary as far as the universal geometric bounds obtained in the next subsection \ref{subsec:geometric-bounds} are concerned. This number satisfies that any orbifold 
$Y$ with $N^2_1$ orthogonal bases $\e^{nm}$, $1 \le n,m \le N_1$, has at least a $\mu_0$-bounded good basis. 
All the results from \cite[Section 6]{Mu-jems} are also valid for orbifolds associated to $N_1=8 \t_0+3$.
\end{remark}

However, we point out that, although the minimal value of $N$ for which we are able to obtain the geometric bounds is 
$N=N_1=130 \, 995$, in order to prove the non-Sasakian property later on, we will need to choose the numbers $n_m$ to be in a certain (quite disperse) family of prime numbers. This choice shall increase considerably the number $N$.

\subsection{Geometric bounds for orbifolds with one $\mu_0$-bounded basis} \label{subsec:geometric-bounds}

Here we collect the results from \cite[Section 6]{Mu-jems}, with some modifications in terminology.
We assume that the orbifold $Y$ has:
\begin{itemize}
    \item at least $N_1^2$ orthogonal bases 
    $\e^{nm}$, $1 \le n,m \le N_1$, and
    \item at least a $\mu_0$-bounded good basis among them.
\end{itemize}  
As shown above, this assumption holds for any choice of $\mu_0 \ge 4 \t_0+2=65\,498$ and $N_1 \ge 2 \mu_0-1=130\,995$.
 
\begin{lemma} {\cite[Section 6]{Mu-jems}}\label{lem:universal-constants}
Denote $K_Y$ the canonical divisor of $Y$.
There exist universal constants:
\begin{itemize}
    \item $\t_1=(18N_1^2-1)^2$ and $\t_2=36N_1^2-1$ so that $-\t_2 \le K_Y^2 \le \t_1$.
    \item $\t_3=\t_2+12$ so that $e(\tilde Y) \le \t_3$ 
    \item $\t_4=4 \t_3-14$ so that any nice curve $C \subset Y$ with $C^2=-m<0$ satisfies $m \le 2g(C)+\t_4$.
\end{itemize}
\end{lemma}

Up to now, we were not able to control which of the curves of a basis $\e^{nm}$ is the positive one. The following result ensures that it is the curve of greater genus for $n$ big enough.

\begin{lemma} \label{lem:n0}
Let $\t_5=(18\mu_0^2-1)^2-72$.
There is some universal number 
\[
n_0=\left \lfloor \sqrt{\frac{\t_2+\t_5}{72}} \right \rfloor +1 ,
\]
such that for any $\mu_0$-bounded good basis $\e^{nm}=(D_1^{nm},D_2^{nm}, D_3^{nm})$ with $n \ge n_0$, the curve $D_1^{nm}$ with genus $9n^2+1$ is the positive curve in $\e^{nm}$.
\end{lemma}

\begin{proof}
Denote $D_i^{nm}=D_i$ to ease notation, and $g_n=9n^2+1$, $g_m=9m^2+1$. Let us assume that $D_2$ is the positive curve. Then 
\begin{equation} \label{eqn:K2}
K^2=\frac{(2g_m-2-m_2)^2}{m_2}-\frac{(2g_n-2+m_1)^2}{m_1}-\frac{(18+m_3)^2}{m_3} \, .
\end{equation}
By Corollary \ref{cor:m1-le-chi1} we have a bound for the positive self-intersection: $1\le D_2^2=m_2< \chi_2=2g_m-2=18m^2 \le 18 \mu_0^2$.
On the other hand by Lemma \ref{lem:universal-constants} we have a bound for the negative self-intersections: $1 \le m_3 \le 20+\t_4$, and $1 \le m_1 \le 2g_n+\t_4$, and we also have the bound $-\t_2 \le K^2$. 

The function $f(x)=\frac{(a-x)^2}{x}$
has derivative $\frac{x^2-a^2}{x}$ so it is decreasing in $(0,a)$, therefore the
maximum of the first term of (\ref{eqn:K2})
is for $m_2=1$. On the other hand,
the function $f(x)=\frac{(a + x)^2}{x}$ has derivative $\frac{x^2-a^2}{x}$, so it is decreasing in $(0,a)$ and increasing in $(a,\infty)$, hence its minimum value is $f(a)=4a$.
This bounds the second and third terms
of (\ref{eqn:K2}). Therefore, 
\begin{align*}
& -\t_2  \le K^2 \le 
(2g_m-3)^2-4(2g_n-2)-4 \cdot 18 \le (18 \mu_0^2-1)^2-72n^2-72=\t_5-72n^2 \\ 
& \implies n \le \sqrt{\frac{\t_2+\t_5}{72}} \, .
\end{align*}
Hence if we take $n \ge n_0=\left \lfloor \sqrt{\frac{\t_2+\t_5}{72}} \right \rfloor +1$,
we get a contradiction.
If we suppose that $D_3$ is the positive curve we get the result by a similar computation.
\end{proof}

The following result gives a lower bound for the positive self-intersection $m_1=(D_1^{nm})^2$ in terms of $n$, i.e. in terms of the genus $g(D_1^{nm})=9n^2+1$. Recall that we denote $(D_2^{nm})^2=-m_2$, $(D_3^{nm})^2=-m_3$ the negative self-intersections.

\begin{lemma}\label{lem:bound-m_1}
Let $n_0$ from Lemma \ref{lem:n0},
and $\e^{nm}$ a $\mu_0$-bounded basis with $n \ge n_0$.
There exists a universal constant 
\[
\t_6=\t_1+(18 \mu_0^2+1)^2+\frac{(38+\t_4)^2}{20+\t_4}\, ,
\]
so that we have the bounds:
\[
18n^2-n\sqrt{18\t_6} \le m_1<18n^2 \, .
\]
\end{lemma}

\begin{proof}
Denote $\e=\e^{nm}$. As $n \ge n_0$, the positive curve of $\e$ is the curve $D_1$ of greater genus.
By Corollary \ref{cor:m1-le-chi1} and Lemma \ref{lem:universal-constants}, the self-intersections satisfy the bounds 
\[
1 \le m_1 < 18n^2 \, , \,  1 \le  m_2 \le 18\mu_0^2+2+\t_4 \, , \,   1 \le m_3 \le 20 + \t_4 \, .
\]
Also, we know that $K^2 \le \t_1$, hence
\begin{align*}
\t_1  \ge K^2 &=\frac{(18n^2-m_1)^2}{m_1}-\frac{(18m^2+m_2)^2}{m_2}- \frac{(18+m_3)^2}{m_3} \\
& \ge 
\frac{(18n^2-m_1)^2}{m_1}-(18 \mu_0^2+1)^2-\frac{(38+\t_4)^2}{20+\t_4}\, ,
\end{align*}
where we have computed the maxima \begin{align*}
&\max \left\{\tfrac{(18m^2+m_2)^2}{m_2}: 2 \le m \le \mu_0, 1 \le m_2 \le 18\mu_0^2+\t_4+2\right\}=(18 \mu_0^2+1)^2 , \\
&\max \left\{\tfrac{(18+m_3)^2}{m_3}: 1 \le m_3 \le\t_4+20\right\}=\frac{(38+\t_4)^2}{20+\t_4} .
\end{align*}
For the first maximum, we note that the maximum is attained at the extrema, i.e.\ for $m_2=1$ or $m_2=18\mu_0^2+\tau_4+2$, hence it is the maximum of the numbers
$(18\mu_0^2+1)^2$ and $\frac{(18\mu_0^2+18\mu_0^2+\tau_4+2)^2}{18\mu_0^2+\tau_4+2}$.
Using that $18 \mu_0^2+2+\t_4 \le 18^2 \mu_0^4$ (which holds true by a large margin for the constants $\mu_0, \t_4$ obtained, see below), we get that the first number is bigger,
hence the claim.

We have obtained that 
\[
\frac{(18n^2-m_1)^2}{m_1} \le \t_1+(18 \mu_0^2+1)^2+\frac{(38+\t_4)^2}{20+\t_4}=\t_6 \, ,
\]
so $m_1^2-(\t_6+36n^2)m_1+18^2 n^4 \le 0$. This, together with the inequality $m_1<18n^2$, implies that 
\[
18n^2+\frac{\t_6}{2}-\sqrt{18n^2\t_6+\frac{\t_6^2}{4}} \le m_1<18n^2 \, .
\]
Now using that $\sqrt{x+y} \le \sqrt{x}+\sqrt{y}$, we get
\[
18n^2-n\sqrt{18\t_6} \le m_1<18n^2\, ,
\]
and this yields the claim.
\end{proof}

\section{Explicit bound for the constant $N$} \label{sec:6}

We start by studying orbifolds $Y$ which 
have at least two distinct $\mu_0$-bounded good bases. This will allow to express the canonical divisor $K_Y$ with respect to different bases, thus obtaining interesting diophantine equalities.

\subsection{Construction of many $\mu_0$-bounded good bases} \label{subsec:two-mu0-bounded}

Let $\mu_0=4 \t_0+1+\ell$ for some $\ell \ge 1$, and let $\cP \subset [\mu_0+1,\infty) \cap \ZZ$ 
be a set with $\# \cP \ge \mu_0-1$. Consider an injective function 
\[
\{m \in \ZZ: 2 \le m \le \mu_0\} \lhook\joinrel\longrightarrow \cP
\]
which assigns to each $m$ a different number $n_m \in \cP$. For simplicity we assume that the numbers $n_m$ are chosen in increasing order. In this way we can construct $\mu_0$-bounded bases $\e(m,\cP)=\e^{n_m m}$ with the numbers $n_m$ lying on a suitable prefixed set $\cP$.

\begin{remark}
The most natural choice for the set is $\cP=\{\mu_0+1,\ldots , 2 \mu_0-1\}$. For this choice of $\cP$ the numbers $n_m$ are the smallest possible, and we shall denote $\e(m)$ the bases thus obtained. However, in Section \ref{sec:non-sasakian}, we shall need to construct bases for a very specific set $\cP$.
\end{remark} 

\begin{lemma} \label{lem:existence-many-mu0-bounded}
Given $\mu_0=4 \t_0+1+\ell$, and given $\cP \subset [\mu_0+1,\infty) \cap \ZZ$, there are at least $\ell$ good $\mu_0$-bounded bases among the bases $\e(m,\cP)$, $2 \le m \le \mu_0$.
\end{lemma}

\begin{proof}
As we saw in Lemma \ref{lem:three-bases}, each singular point $p$ of the orbifold $Y$ lies in at most four of the bases $\e(m,\cP)$. Since there are at most $\t_0$ singular points in $Y$, there are at most $4 \t_0$ bases among the bases $\e(m,\cP)$ that are not good, hence at least $\ell$ of these bases are good.
\end{proof}
If we take $\ell=2$ we get the condition for $Y$ to admit two good $\mu_0$-bounded bases, so in this subsection we shall assume that $Y$ has
\begin{itemize}
    \item at least $N_1^2$ orthogonal bases $\e^{nm}$, $1 \le n,m \le N_1$, and
    \item at least two $\mu_0$-bounded bases among them.
\end{itemize}
If $\mu_0=4 \t_0+3=65\,499$ and $N_1 \ge 2 \mu_0-1=130\,997$, the above holds.

\begin{lemma} \label{lem:m1/m'1-eps}
Given $\eps \in (0,1)$ there exists a number $n_1>0$, namely
\[
n_1=\frac{1+\eps}{\eps} \sqrt{\frac{\t_6}{18}} \, ,
\]
so that for any pair of good $\mu_0$-bounded bases $\e=\e^{nm}$, $\e'=\e^{n'm'}$ with $n, n' \ge  \max\{n_0,n_1\}$, the positive self-intersections $m_1$, $m'_1$ satisfy that
\[
\frac{m_1}{m'_1} \in \left((1-\eps) \frac{n^2}{n'^2}, (1+\eps) \frac{n^2}{n'^2}\right).
\]
\end{lemma}

\begin{proof}
Using the bounds from Lemma \ref{lem:bound-m_1} we have
\[
\frac{m_1}{m'_1} \in \left( \frac{18n^2-\sqrt{18 \t_6} n}{18n'{}^2}, \frac{18n^2}{18n'{}^2-\sqrt{18 \t_6} n'} \right) \, .
\]
On the other hand,
\[
(1-\eps) \frac{n^2}{n'{}^2}<\frac{18n^2-\sqrt{18 \t_6} n}{18n'{}^2} \quad \text{ occurs when } \quad n>\frac{1}{\eps} \sqrt{\frac{\t_6}{18}} \, ,
\]
\[
\frac{18n^2}{18n'^2-\sqrt{18 \t_6} n'}<
(1+\eps) \frac{n^2}{n'{}^2} \quad \text{ occurs when } \quad n'>\frac{1+\eps}{\eps} \sqrt{\frac{\t_6}{18}}\, ,
\]
\medskip and this yields the result.
\end{proof}

The next result gives an arithmetic relation between the genus of the positive curve $D_1^{nm}$ (in terms of $n$) and the self-intersection $m_1=(D_1^{nm})^2$.

\begin{proposition} \label{prop:denominators-self-intersection}
Take $n_0$ as in Lemma \ref{lem:n0}. There exists a universal $R \in \ZZ_{>0}$, given by
\[
R=18^2 \lcm(2,3,4,\dots, \t_4+18\mu_0^2+2),
\]
so that for any pair of good $\mu_0$-bounded bases $\e^{nm}$, $\e^{n'm'}$ with $n, n' \ge n_0$, $n \ne n'$, we have
\[
R \left(\frac{n^4}{m_1}-\frac{n'{}^4}{m'_1}\right) \in \ZZ,
\]
where $m_1=(D_1^{nm})^2\, , \, m'_1=(D_1^{n'm'})^2 >0$ are the positive self-intersections.
\end{proposition}

\begin{proof}
Let us denote $\e=\e^{nm}=(D_1,D_2,D_3)$ with genera $g_n=9n^2+1$, $g_m=9m^2+1$, $10$, and $\e'=\e^{n'm'}=(D'_1,D'_2,D'_3)$ with genera $g_{n'}=9n'{}^2+1$, $g_{m'}=9m'{}^2+1$, $10$.
By $\mu_0$-boundedness we have that $37 \le g_m, g_{m'} \le 9 \mu_0^2+1$. Since $n,n'\ge n_0$ we know that $D_1, D'_1$ are the positive curves. We express the canonical divisor $K_Y$ in both bases and compute its square:
\begin{align*}
K_Y^2& =\frac{(18n^2-m_1)^2}{m_1}-\frac{(18m^2+m_2)^2}{m_2}- \frac{(18+m_3)^2}{m_3} \\
&=\frac{(18n'^2-m'_1)^2}{m'_1}-\frac{(18m'^2+m'_2)^2}{m'_2}- \frac{(18+m'_3)^2}{m'_3} \, .
\end{align*}
Subtracting and expanding the squares we deduce that
\[
18^2 \left(\tfrac{n^4}{m_1}- \tfrac{n'^4}{m'_1}\right)=  36(n^2-n'^2)+m_1-m'_1+\tfrac{(18m^2+m_2)^2}{m_2}-\tfrac{(18m'^2+m'_2)^2}{m'_2}+ \tfrac{(18+m_3)^2}{m_3}- \tfrac{(18+m'_3)^2}{m'_3}\, .
\]
On the other hand, we have the bounds $m_2, m'_2 \le \t_4+18 \mu_0^2+2$, and $m_3, m'_3 \le \t_4+20$ from Lemma \ref{lem:universal-constants}. Then we take $R=18^2 \lcm(2,3,4,\dots, \t_4+18 \mu_0^2+2)$ and get the result.
\end{proof}

\begin{remark} \label{rem:R-estimation}
The number $R$ is given by:
\[
R=18^2 \prod_p p^{e_p} \, ,
\]
where $p$ ranges over all primes with $2 \le p \le \t_4+18\mu_0^2+2$, and $e_p=\lfloor \frac{\log(\t_4+18\mu_0^2+2)}{\log p} \rfloor$.
\end{remark}

With the additional assumption that $n, n'$ are prime numbers in Proposition \ref{prop:denominators-self-intersection}, we get the following:

\begin{lemma} \label{lem:m1/m'1-divisors}
Let $n_0$ and $R$ as in Proposition \ref{prop:denominators-self-intersection}. Denote $\cD_R=\{d_1, \dots, d_t\}$ the set of divisors of $R$. Suppose that there exists a pair of $\mu_0$-bounded good bases $\e^{nm}, \e^{n'm'}$ with $n, n' \ge n_0$ two distinct prime numbers. Then the positive self-intersections $m_1=(D_1^{nm})^2, m'_1=(D_1^{n'm'})^2$ satisfy
\[
\frac{m_1}{m'_1} \in \left\{\frac{d_i}{d_j} \frac{n^\b}{n'{}^{\g}}: (d_i,\b), (d_j,\g) \in \left(\cD_R \times \{0,1\}\right) \cup \left(\{1,\dots, 17\} \times \{2\}\right)  \right\} .
\]

\end{lemma}

\begin{proof}
Let $d=\gcd(m_1,m'_1)$ so that $m_1=da$, $m'_1=da'$, with $a,a'$ coprime. By Proposition \ref{prop:denominators-self-intersection}, we know that $R(\frac{n^4}{m_1}-\frac{n'{}^4}{m'_1}) \in \ZZ$, so after multiplying by $daa'$ we get
\[
a'Rn^4-aRn'{}^4 \in daa'\ZZ \implies a |a'Rn^4 \, , \,  a'|aRn'{}^4 \implies a |Rn^4 \, , \,  a' |Rn'{}^4\, .
\]

As $n,n'$ are prime numbers, this yields that $a=d_in^{\b}$, $a'=d_jn'{}^{\g}$ for some $0\le \b,\g \le 4$ and $d_i, d_j \in \cD_R$. On the other hand we know by Corollary \ref{cor:m1-le-chi1} that the positive self-intersections satisfy the bounds:
\[
18n^2>m_1=da=dd_in^{\b} \, , \qquad 
18n'{}^2>m'_1=da'=dd_jn'{}^{\g}\, .
\]
As $n,n'\ge n_0>18$ we exclude the possibility that $\b$ or $\g$ are $\ge 3$. Moreover if $\b=2$ it follows that $d_i<18$, and if $\g=2$ then $d_j<18$. Hence $\frac{m_1}{m'_1}=\frac{a}{a'}$ satisfies the statement.
\end{proof}

\subsection{The non-Sasakian property} \label{sec:non-sasakian}

Now we aim to see that if we take a sufficiently large number $N$, then we can arrange the conditions from Lemmas \ref{lem:m1/m'1-eps} and \ref{lem:m1/m'1-divisors} to be incompatible unless $m_1=d n^2$, $m'_1=d n'{}^2$, for some $d<18$. 
We thus make the same initial assumptions on $Y$ as in Subsection \ref{subsec:two-mu0-bounded}, but later on we will need to increase the value of $N$.

\medskip

\noindent \textbf{A disperse sequence of primes:}
In order to gain control on the values of $n,n'$ of the two $\mu_0$-bounded good bases, we construct a suitable set $\cP$ as follows. Let 
\[
\eps_0=\tfrac{1}{17}=\min\left\{\left|1-\tfrac{k_i}{k_j}\right|: 1 \le k_i, k_j \le 17, i \ne j\right\},
\]
and consider the constants $n_0$ from Lemma \ref{lem:n0}, $R$ from Proposition \ref{prop:denominators-self-intersection}, and $n_1=n_1(\eps_0)$ from Lemma \ref{lem:m1/m'1-eps}. We construct a sequence of primes recursively as follows:
\begin{equation} \label{eqn:cP}
\begin{cases}
    \nu_0=\text{ first prime greater than }  \max \left\{R(1+\eps_0),n_0,n_1, \sqrt{18 \t_6}\, \right\} , \\ 
    \nu_{i+1}=\text{ first prime greater than } {\displaystyle \frac{R \n_i^2}{1-\eps_0}}\, .
\end{cases}
\end{equation}
We denote $\cP=\{\nu_i: i \ge 0\}$ this sequence of prime numbers. Notice that our estimation of $R$ from Remark \ref{rem:R-estimation} gives a very large number, in particular we have $R  \gg n_0, n_1, \sqrt{18 \t_6}$, so in the definition of $\nu_0$, 
it is not necessary to take the maximum.

\medskip

Recall the construction of $\mu_0$-bounded bases of type $\e(m,\cP)$: for each $m \in [2,\mu_0]$ one chooses $n_m \in \cP$ in increasing order, and $\e(m,\cP)=\e^{n_m m}$.

\begin{proposition} \label{prop:non-existence-Y-1}
Let $\cP$ be the sequence of primes (\ref{eqn:cP}). There does not exist a Kähler orbifold $Y$ with a pair of $\mu_0$-bounded bases $\e(m,\cP)=\e^{n_mm}$, $\e(m',\cP)=\e^{n_{m'}m'´}$.
\end{proposition}

\begin{proof}
Let $n=n_m$, $n'=n_{m'} \in \cP$, and assume $n>n'$, that is $n=\nu_i$, $n'=\nu_{j}$ for some $i>j$. On the one hand, by Lemma \ref{lem:m1/m'1-divisors},
\[
\frac{m_1}{m'_1} \in  \left\{\frac{d_i}{d_j} \frac{n^\b}{n'^{\g}}: (d_i,\b), (d_j,\g) \in ( \cD_R \times \{0,1\}) \cup (\{1,\dots, 17\} \times \{2\} ) \right\} \, ,
\]
while, on the other hand, by Lemma \ref{lem:m1/m'1-eps} and the construction of $\cP$,
\[
\frac{m_1}{m'_1} \in \left((1-\eps_0) \frac{n^2}{n'{}^2} \, , \, (1+\eps_0) \frac{n^2}{n'{}^2}\right).
\]
Let us see that these two conditions are possible only if \ $\displaystyle \frac{m_1}{m'_1}=\frac{n^2}{n'{}^2}$. 

\noindent \textbf{Step 1:} First we discard the case that some of the exponents $\b, \g \in \{0,1\}$. To begin, note that all quotients $s=\frac{d_i}{d_j}\in [\frac1R,R]$.
Also, from the definition of $\cP$ we have that $n \ge R \frac{n'{}^2}{1-\eps_0}$, or equivalently, 
$(1-\eps_0)\frac{n}{n'{}^2} \ge R$. So
\[
\frac{m_1}{m'_1}>(1-\eps_0)\frac{n^2}{n'{}^2} \ge  n R,
\]
which is bigger than any of
the expressions $s, s n, s\frac1{n'}, s\frac{n}{n'}, s\frac1{n'^2}, s\frac{n}{n'{}^2}$. 
Therefore $\beta=2$.

Also, since $n' \ge R(1+\eps_0)$, we deduce that
\[
\frac{m_1}{m'_1} < (1+\eps_0)\frac{n^2}{n'{}^2}\leq \frac1R \frac{n^2}{n'}\, ,
\]
which is smaller than any of the
expressions $s\frac{n^2}{n'}, s n^2$. Hence $\gamma=2$.

\medskip

\noindent \textbf{Step 2:} It follows that $m_1=d_in^2$, $m'_1=d_jn'{}^2$, for some $1 \le d_i, d_j \le 17$, so $\frac{m_1}{m'_1}=\frac{d_i}{d_j} \frac{n^2}{n'{}^2}$.
By the choice of $\eps_0$, none of the quotients $\frac{d_i}{d_j}$ lie in $(1-\eps_0,1+\eps_0)$  unless $d_i=d_j$, hence we conclude that
\[
 \frac{m_1}{m'_1}=\frac{n^2}{n'{}^2}\, .
\]
Therefore $m_1=d n^2$, $m_1'=d n'{}^2$, for some $1 \le d \le 17$. In particular $m_1 \le 17n^2$, $m'_1 \le 17 n'{}^2$. However, by the bounds on Lemma \ref{lem:bound-m_1} we have
\begin{align*}
   18n^2-n\sqrt{18\t_6} &\le m_1<18n^2 ,\\
   18n'^2-n'\sqrt{18\t_6} &\le m'_1<18n'^2 ,
\end{align*}
and this gives a contradiction with the inequalities 
\[
17n^2< 18 n^2-n\sqrt{18\t_6} \,  , \quad 17n{}'^2< 18 n'{}^2-n'\sqrt{18\t_6} \, ,
\]
which hold because $n, n' \in \cP$, so in particular $n, n' >  \sqrt{18 \t_6}$.
\end{proof}

\begin{proposition} \label{prop:non-existence-Y-2}
Consider $\mu_0=4 \t_0+3$, and 
\[
N \ge  \nu_{\mu_0-2} = \text{$(\mu_0-2)$-th term of the sequence of primes $\cP$.}
\]
Then there does not exist a Kähler orbifold $Y$ with $N^2$ orthogonal bases $\e^{nm}$, with $1 \le n,m \le N$.
\end{proposition}

\begin{proof}
If there exists such an orbifold $Y$, then $Y$ satisfies all the properties from Subsection \ref{subsec:two-mu0-bounded}. Also, from Lemma \ref{lem:existence-many-mu0-bounded}, among the bases $\e(m,\cP)=\e^{n_mm}$, for $2 \le m \le \mu_0$, constructed by choosing 
\[
n_2=\nu_0 , \, \ldots, \,n_{\mu_0}=\nu_{\mu_0-2}\, ,
\]
there would be at least two $\mu_0$-bounded good bases. The non-existence of such an $Y$ follows from Proposition \ref{prop:non-existence-Y-1}.
\end{proof}

The following follows immediately after all this work. This shows a explicit bound for the number $N$ from Theorem \ref{thm:main-jems} and Corollary \ref{cor:main-1}. 

\begin{theorem} \label{th:main-sec6}
Let $\mu_0=4 \t_0+3$. For a choice of $N \ge \nu_{\mu_0-2}$, the $K$-contact Smale-Barden manifolds $M$ constructed in \cite[Theorem 3.4]{Mu-jems} do not admit any Sasakian structure.
\end{theorem}

\subsection{Estimation of the constants}

As we have seen in Theorem \ref{th:main-sec6}, we can ensure that the $K$-contact Smale-Barden manifold $M$ from \cite[Theorem 3.4]{Mu-jems} cannot be Sasakian if we take a number $N \ge \nu_{\mu_0-2}$. 
In this section we shall compute the constants involved in the construction of $M$, giving bounds for the value $N=\nu_{\mu_0-2}$.
This will complete the proof of Theorem 
\ref{th:main-2}.

\subsubsection{Explicit universal constants}

First we need to give an explicit expression for the constants involved in the definition of $\cP$.
These constants are:
\begin{itemize}
    \item We have $\t_0=16\,374$, the bound for the singular points of $Y$, as long as we take $N \ge 11$ (see Remark \ref{rem:5.6}).
    \item We can take $\mu_0=4 \t_0+3=65\,499$ to ensure the existence of two $\mu_0$-bounded basis for any election of $n_m$, $2 \le m \le \mu_0$.
    \item The number $N_1=2 \mu_0-1=130\,997$ is the lower bound for the value of $N$ in order to have two $\mu_0$-bounded bases constructed by choosing $n_m= \mu_0+m-1$
    (see discussion after Lemma \ref{lem:existence-many-mu0-bounded}).
\end{itemize}
Once given the number $N_1$, we have the universal geometric constants for $N \ge N_1$ and orbifolds $Y$ with bases $\e^{nm}$, $1 \le n,m \le N$, given as follows, according to
Section \ref{sec:5}:
\begin{itemize}
    \item $\t_1=(18N_1^2-1)^2$ and $\t_2=36N_1^2-1$ so that $-\t_2 \le K_Y^2 \le \t_1$.

    \item $\t_4=144N_1^2+30$ so that $-C^2 \le 2g(C)+\t_4$ for any nice curve.

    \item The number 
    \[
    n_0=\left \lfloor \sqrt{\tfrac{36N_1^2-1+(18\mu_0^2-1)^2-72}{72}} \right \rfloor +1 ,
    \]
    to ensure $D_1^{nm}$ is the positive curve.
\end{itemize}
We also have:
\begin{itemize}
    \item The number $\t_6=(18N_1^2-1)^2+(18 \mu_0^2+1)^2+\tfrac{(144N_1^2+68)^2}{144N_1^2+50}$,
    from Lemma \ref{lem:bound-m_1}.
    \item The number $n_1=18 \sqrt{\frac{\t_6}{18}}=\sqrt{18 \t_6}$, from Lemma \ref{lem:m1/m'1-eps}, with $\eps=\frac{1}{17}$.
    \item The number $R=18^2 \lcm(2,3,4,\dots, \t_4+18\mu_0^2+2)$, from Proposition \ref{prop:denominators-self-intersection}.
\end{itemize}

If we take $\mu_0=65\,499$, $N_1=2 \mu_0-1=130\,997$, we get:
\begin{align*}
& \t_1= 95\,409\,234\,125\,818\,504\,369 \,921 \, ,  &&\t_2=617\,767\,704\,323 \, , \\  & \t_4=2\,471\,070\,817\,326 \, ,
 &&n_0=9\,100\,716\,714 \, , \\ &\t_6=
 101\,372\,493\,346\,292\,180\,583\,664 \, , && n_1=1\,350\,816\,375\,469.
\end{align*}
We have rounded up the numbers $\t_6$ and $n_1$, since in our definition they are not
integers.

The number $R$ is
\begin{align*}
& R=18^2 \lcm(2,3,4,\dots, 2\,548\,292\,959\,346)= 18^2 \prod_{p \text{ prime}} p^{e_p} \, , \\
& \text{where $2 \le p \le 2\,548\,292\,959\,346$, and $e_p=\lfloor \tfrac{\log(2\,548\,292\,959\,346)}{\log p} \rfloor \leq \lfloor \tfrac{29}{\log p} \rfloor$} \, .
\end{align*}
The number $R$ is by far the bigger number among those computed here. A very rough estimate gives:
\[
R=18^2 \prod_{p \text{ prime}} p^{e_p}  > 18^2 \cdot 210 \cdot \prod_{p \ge 11}^{\,\,\,\le 10^{12}} p \, \, > 18^2 \cdot 210 \cdot 11^{\Pi(10^{12})-4} = 18^2 \cdot 210 \cdot 11^{37\,607\,912\,014} ,
\]
where $\Pi$ is the prime counting function, and $\Pi(10^{12})=37\,607\,912\,018$. Hence the number $R$ has more than thirty thousand million decimal ciphers. There is also a trivial upper bound for $R$ by a factorial, so we have the estimates
\begin{equation} \label{eq:estimate-R}
18^2 \cdot 420 \cdot 11^{37\,607\,912\,014} < R < 18^2 \cdot 2\,548\,292\,959\,346\,!
\end{equation}

\subsubsection{Estimation for $N$.}

The sequence of primes $\cP$ is given by:
\[
\begin{cases}
    \nu_0=\text{ first prime greater than } \frac{18}{17} R, \\[5pt] 
    \nu_{i+1}=\text{ first prime greater than } \frac{17}{16} R \n_i^2\, .
\end{cases}
\]
Let us give a lower and upper bound for $N=\nu_{\mu_0-2}=\nu_{65\,497}$.

\noindent {\bf Lower bound.}
As $\nu_0 \ge \frac{18}{17} R$ and
$\nu_{i+1} \geq \frac{17}{16} R \n_i^2$, we have by induction that 
\[
\nu_k \ge \tfrac{16}{17} \left(\tfrac{18}{16}\right)^{2^k} R^{2^{k+1}-1} \, .
\]
Hence 
\[
N \ge \nu_{65\,499} \ge \tfrac{16}{17} \left(\tfrac{18}{16}\right)^{2^{65\,499}} R^{2^{65\,499}-1} \, .
\]
This gives an idea on how large we need to take the number $N$ to ensure that the manifold $M$ is not Sasakian.

\noindent{\bf Upper bound.}
In order to give an upper bound for $N$, we need some result ensuring the existence of primes in a given interval. We shall denote $r>1$ a number such that for any $n \ge R$ there exists a prime with $n < p < rn$. There are various results giving some values of $r$, for instance:
\begin{itemize}
    \item Bertrand's postulate, proved by Tchebychev \cite{Tchebychev}: Given an integer $n>1$, there exists a prime number $p$ in the interval so that $n <p<2n$. 
    \medskip
    \item A sharper result by Schoenfeld \cite{Schoenfeld}: Given an integer $n>2\,010\,760$, there exists a prime $p$ so that $n<p<(1+\frac{1}{16\,597})n$.
    \medskip
    \item A result by Dusart \cite[Proposition 5.4]{Dusart} that works better for big $n$:
    If $n \ge 89\,693$, there is a prime $p$ with $n<p\le \big(1+\frac{1}{\log^3 n}\big) n$.
\end{itemize}
Of course, the above shows that $r=2$ and $r=1+\frac{1}{16\,597}$ are valid.
Having into account how large $R$ is, the third result gives the best value of $r$: we can take $r=1+\frac{1}{\log^3 R}$. Looking at the estimates for $R$ in \eqref{eq:estimate-R}, and using that $11>e^2$, we get $\log R > 7 \cdot 10^{10}$, so $\log^3 R > 10^{32}$. 
We deduce:
\begin{lemma} \label{lem:r}
Let $r= 1+10^{-32}$. For any $n \ge R$, there exists a prime $p$ with $n < p < rn$.
\end{lemma}

Using this result, we get that 
$\nu_0$ satisfies $\nu_0 < \frac{18}{17}r R$,
and that 
$\nu_{i+1} \leq \frac{17}{16} rR \n_i^2$. By induction we get
\[
\nu_k \leq \tfrac{16}{17} \left(\tfrac{18}{16}\right)^{2^k} (rR)^{2^{k+1}-1} \, .
\]
\medskip
Taking into account that the minimum value of $N$ is $\nu_{65\,499}$, we obtain:

\begin{proposition}
For any number $N$ such that
\[
N \ge 
\tfrac{16}{17} \left(\tfrac{18}{16}\right)^{2^{65\,499}} (rR)^{2^{65\,500}-1} \, ,
\]
the $K$-contact Smale-Barden manifold $M$ from \cite{Mu-jems} with $N^2$ orthogonal basis, cannot admit a Sasakian structure.
\end{proposition}

Recall that we can take $r=1+10^{-32}$ by Lemma \ref{lem:r}, and that explicit bounds for $R$ are given above in equation \eqref{eq:estimate-R}. This completes
the proof of Theorem \ref{th:main-2}.

\section{Some sharper results}

In the case of the K-contact manifolds from Theorem \ref{th:main-0}, it is possible to take advantage of the curves $A_k$, $k \ge 1$, to get sharper bounds. Let $Y$ be the orbifold in (\ref{eqn:MtoY}), that 
is the base of a Seifert bundle for a manifold
$M$ as in
Theorem \ref{th:main-0}. Recall that the isotropy locus of the orbifold consists of isolated singular points $p \in P$, and some complex curves given
by the triples (\ref{eqn:enmk}) for 
$1\leq n\leq N_1$, $1\leq m\leq N_2$, 
$1\leq k\leq K_0$.

\subsection{Bound for the number of singular points}

First, we can bound the number of singular points as follows. According tu Proposition \ref{prop:bound-singular}, we need three bases with all curves distinct. This can be ensured now by means of the genera, without the need of the concept of proj-equivalence of bases. We take three bases $\e^{nmk}$ with $k<m<n$ as follows:
\begin{itemize}
    \item For $k=1,m=2, n=3$ we have genera $10, 37, 82$.
    \item For $k=4, m=5, n=6$ we have genera $145, 226, 325$.
    \item For $k=7, m=8, n=9$ we have genera $442, 577, 730$.
\end{itemize}
Now, using Proposition \ref{prop:bound-singular} we get that the number of singularities is bounded by 
$$
\# P \le \hat \t_0=-6 +4 (10+37+82+145+226+325+442+577+730)=10 \, 290 \, .
$$
We conclude the following.

\begin{proposition} \label{prop:hat-tau_0}
Any Kahler orbifold $Y$ having three bases $\e^{nmk}$ for the values 
\[
(n,m,k)= (3,2,1), (6,5,4), (9,8,7)
\]
has at most $\hat \tau_0=10\, 290$ singular points.
\end{proposition}

\subsection{Construction of suitable good bases.}

We start with a concept that shall be useful in the future.

\begin{definition}
    Let $\k_0<\mu_0$ be positive integers. We say that a basis $\e^{nmk}$ is $(\k_0,\m_0)$-bounded if $k \le \k_0$ and $m \le \mu_0$.
\end{definition}
We can construct suitable such bases as follows. Take a number $\k_0$, which will be the number of bases constructed. 
For each $k$ with $1 \le k \le \k_0$ we take $m_k=\k_0+k$. We now take numbers $n_k$ in a subset $\cP \subset [2\k_0+1,+\infty) \cap \ZZ$ of cardinality $\k_0$. We denote $\hat \e(k,\cP)=\e^{n_km_kk}$ the bases thus constructed. Note that all the curves in these bases have distinct genera. Moreover, the genus of two of the curves in the
bases $\hat \e(k,\cP)$ is uniformly bounded by $\k_0$ and $2 \k_0$ respectively, so these are $(\k_0,\m_0)$-bounded bases with $\m_0=2\k_0$.

As a particular case, note that the smallest possible choice of the numbers $n_k$ is $n_k=2\k_0+k$, corresponding to $\cP=[2\k_0+1,3\k_0+1] \cap \ZZ$. We denote $\hat \e(k)=\e^{n_km_kk}$ the bases corresponding to this choice. 

The fact that a singular point of $Y$ can be in at most two isotropy surfaces readily implies the following.

\begin{lemma} \label{lem:above}
Any singular point $p \in P$ is in at most two of the bases $\hat \e(k)$.
Let $\k_0=2 \hat\t_0+\ell$, for $\ell\ge 1$. Among the bases $\hat \e(k,\cP)$ constructed above for $1 \le k \le \k_0$, there are at least $\ell$ good bases that are $(\k_0,\mu_0)$-bounded with $\mu_0=2\k_0$.
\end{lemma}

\subsection{Universal geometric bounds} \label{subsec:geometric-bounds-hat}

We aim to deduce universal geometric bounds for orbifolds with some good $(\k_0,\mu_0)$-bounded basis. 

We take $\k_0=2 \hat \t_0+2$, and we deduce from Lemma \ref{lem:above} that at least two of the bases $\hat \e(k)$, for $1 \le k \le \k_0$, are good bases which are $(\k_0,2\k_0)$-bounded, and we have $n_k=2\k_0+k$. Define
\[
\hat N_1=\max \{n_k: k \in [1,\k_0]\}=3 \k_0=6 \hat \t_0+6=61\,746 \, .
\]
We shall consider an orbifold which has at least two $(\k_0,\m_0)$-bounded bases $\e^{nmk}$. This is ensured if we assume that $Y$ has at least the bases $\hat \e(k)$ for $1 \le k \le \k_0=2 \hat \t_0+2$; in this case we have $\mu_0=2\k_0$, and $9\hat N_1^2+1$ is the maximum of the genera of such curves.

\begin{lemma} \label{lem:universal-constants-hat}
Let $Y$ be a cyclic Kähler orbifold having the bases $\hat \e(k)$ for $1 \le k \le \k_0=2\hat \t_0+2$. There exist universal constants:
\begin{itemize}
    \item $\hat \tau_1=(18 \hat N_1^2-1)^2$ and $\hat \tau_2=36 \hat N_1^2-1$ so that $-\hat \tau_2 \le K_Y^2 \le \hat \tau_1$.
    \item $\hat \tau_3=12+\hat \tau_2$ so that $e(Y) \le \hat \tau_3$. 
    \item $\hat \tau_4=4 \hat \tau_3-14$ so that for any nice curve $C \subset Y$, if $C^2=-m<0$ then $m \le 2g(C)+\hat \tau_4$.
\end{itemize}
\end{lemma}
\begin{proof}
The proof of these facts follows the same lines as in \cite{Mu-jems}.
\end{proof}

The following result guarantees that the curve with greater genus is the positive one in a $(\k_0,\m_0)$-bounded basis for $n$ big enough.

\begin{lemma} \label{lem:hat-n0}
Let $\hat \t_5=(18 \m_0^2-1)^2-72$.
There is some universal number 
\[
\hat n_0=\left \lfloor \sqrt{\frac{\hat \t_2+\hat \t_5}{72}} \right \rfloor +1 ,
\]
such that for any $(\k_0,\m_0)$-bounded good basis $\e^{nmk}$ with $n \ge \hat n_0$, the curve $D_1^{nmk}$ with genus $9n^2+1$ is the positive curve in $\e^{nmk}$.
\end{lemma}
\begin{proof}
The same proof as in Lemma \ref{lem:n0} applies.
\end{proof}

\begin{lemma} \label{lem:bound-m1-hat}
Let $\hat n_0$ from Lemma \ref{lem:hat-n0},
and $\e^{nmk}$ a $(\k_0,\mu_0)$-bounded basis with $n \ge \hat n_0$.
There exists a universal constant 
\[
\hat \t_6=\hat \t_1+(18 \mu_0^2+1)^2+(18 \k_0^2+1)^2
\]
so that we have the bounds:
\[
18n^2-n\sqrt{18 \hat \t_6} \le m_1<18n^2 \, .
\]
\end{lemma}

\begin{proof}
Denote $\e=\e^{nmk}$. As $n \ge \hat n_0$, the positive curve of $\e$ is the curve $D_1$ of greater genus.
By Corollary \ref{cor:m1-le-chi1} and Lemma \ref{lem:universal-constants-hat}, the self-intersections satisfy the bounds 
\[
1 \le m_1 < 18n^2 \, , \,  1 \le  m_2 \le 18\mu_0^2+2+\hat \t_4 \, , \,   1 \le m_3 \le 18\k_0^2+2+\hat \t_4 \, .
\]
Also, we know that $K_Y^2 \le \hat \t_1$, hence
\begin{align*}
\hat \t_1  \ge K_Y^2 &=\frac{(18n^2-m_1)^2}{m_1}-\frac{(18m^2+m_2)^2}{m_2}- \frac{(18k^2+m_3)^2}{m_3} \\
& \ge 
\frac{(18n^2-m_1)^2}{m_1}-(18 \mu_0^2+1)^2-(18 \k_0^2+1)^2 \, ,
\end{align*}
In the last inequality we have computed the maxima \begin{align*}
&\max \left\{\tfrac{(18m^2+m_2)^2}{m_2}: 1 \le m \le \mu_0, 1 \le m_2 \le 18\mu_0^2+\hat \t_4+2\right\}=(18 \mu_0^2+1)^2 , \\
&\max \left\{\tfrac{(18k^2+m_3)^2}{m_3}: 1 \le k \le \k_0, 1 \le m_3 \le 18 \k_0^2+ \hat \t_4+2\right\}=(18 \k_0^2+1)^2 \, .
\end{align*}
These maxima are attained at the boundary, at the points $(m,m_2)=(\mu_0,1)$ and $(k,m_3)=(\k_0,1)$.

Summing up, we have obtained that 
\[
\frac{(18n^2-m_1)^2}{m_1} \le \hat \t_1+(18 \mu_0^2+1)^2+(18 \k_0^2+1)^2= \hat \t_6 \, ,
\]
so $m_1^2-(\hat \t_6+36n^2)m_1+18^2 n^4 \le 0$. Now the proof follows as in Lemma \ref{lem:bound-m_1}.
\end{proof}

\begin{lemma} \label{lem:m1/m'1-eps-hat}
Given $\eps \in (0,1)$ there exists a number $\hat n_1>0$, namely
\[
\hat n_1=\frac{1+\eps}{\eps} \sqrt{\frac{\hat \t_6}{18}} \, ,
\]
so that for any pair of good $(\k_0,\mu_0)$-bounded bases $\e=\e^{nmk}$, $\e'=\e^{n'm'k'}$ with $n, n' \ge  \max\{\hat n_0, \hat n_1\}$, the positive self-intersections $m_1$, $m'_1$ satisfy that
\[
\frac{m_1}{m'_1} \in \left((1-\eps) \frac{n^2}{n'{}^2}, (1+\eps) \frac{n^2}{n'{}^2}\right).
\]
\end{lemma}

\begin{proof}
The same proof from Lemma \ref{lem:bound-m_1} applies here, using the bounds from Lemma \ref{lem:bound-m1-hat}.
\end{proof}

\begin{proposition} \label{prop:R-hat}
Take $\hat n_0$ as in Lemma \ref{lem:hat-n0}. There exists a universal $\hat R \in \ZZ_{>0}$, given by
\[
\hat R=18^2 \lcm(2,3,4,\dots, \hat \t_4+18\mu_0^2+2),
\]
so that for any pair of good $(\k_0,\mu_0)$-bounded bases $\e^{nmk}$, $\e^{n'm'k'}$ with $n, n' \ge \hat n_0$, $n \ne n'$, we have
\[
\hat R \left(\frac{n^4}{m_1}-\frac{n'{}^4}{m'_1}\right) \in \ZZ,
\]
where $m_1=(D_1^{nmk})^2\, , \, m'_1=(D_1^{n'm'k'})^2 >0$ are the positive self-intersections.
\end{proposition}

\begin{proof}
Let us denote $\e=\e^{nmk}=(D_1,D_2,D_3)$ with genera $g_n=9n^2+1$, $g_m=9m^2+1$, $g_k=9k^2+1$, and $\e'=\e^{n'm'k'}=(D'_1,D'_2,D'_3)$ with genera $g_{n'}=9n'{}^2+1$, $g_{m'}=9m'{}^2+1$, $g_{k'}=9k'{}^2+1$.
By $(\k_0,\mu_0)$-boundedness we have that 
\[
1 \le g_k, g_{k'} \le 9 \k_0^2+1 < g_m, g_{m'} \le 9 \mu_0^2+1 \, .
\]
Since $n,n'\ge \hat n_0$ we know that $D_1, D'_1$ are the positive curves. We express the canonical divisor $K_Y$ in both bases and compute its square:
\begin{align*}
K_Y^2& =\frac{(18n^2-m_1)^2}{m_1}-\frac{(18m^2+m_2)^2}{m_2}- \frac{(18+m_3)^2}{m_3} \\
&=\frac{(18n'{}^2-m'_1)^2}{m'_1}-\frac{(18m'{}^2+m'_2)^2}{m'_2}- \frac{(18+m'_3)^2}{m'_3} \, .
\end{align*}
Subtract and expand the squares so
\[
18^2 \left(\tfrac{n^4}{m_1}- \tfrac{n'{}^4}{m'_1}\right)=  36(n^2-n'{}^2)+m_1-m'_1+\tfrac{(18m^2+m_2)^2}{m_2}-\tfrac{(18m'{}^2+m'_2)^2}{m'_2}+ \tfrac{(18+m_3)^2}{m_3}- \tfrac{(18+m'_3)^2}{m'_3}\, .
\]
On the other hand, we have the bounds $m_2, m'_2 \le \hat \t_4+18 \mu_0^2+2$, and $m_3, m'_3 \le \hat \t_4+18 \k_0^2+2$ from Lemma \ref{lem:universal-constants-hat}. As $\k_0 < \mu_0$, we take $R=18^2 \lcm(2,3,4,\dots, \hat \t_4+18 \mu_0^2+2)$ and get the result.
\end{proof}

\begin{lemma} \label{lem:m1/m'1-divisors-hat}
Let $\hat n_0$ and $\hat R$ as in Proposition \ref{prop:denominators-self-intersection}. Denote $\cD_{\hat R}=\{d_1, \dots, d_t\}$ the set of divisors of $\hat R$. Suppose that there exists a pair of $(\k_0,\mu_0)$-bounded good bases $\e^{nmk}, \e^{n'm'k'}$ with $n, n' \ge n_0$ two distinct prime numbers. Then the positive self-intersections $m_1=(D_1^{nmk})^2, m'_1=(D_1^{n'm'k'})^2$ satisfy
\[
\frac{m_1}{m'_1} \in \left\{\frac{d_i}{d_j} \frac{n^\b}{n'{}^{\g}}: (d_i,\b), (d_j,\g) \in \left(\cD_{\hat R} \times \{0,1\}\right) \cup \left(\{1,\dots, 17\} \times \{2\}\right)  \right\} .
\]
\end{lemma}

\begin{proof}
The same proof from Lemma \ref{lem:m1/m'1-divisors} applies here.
\end{proof}

\subsection{The non-Sasakian property} 

Let $\eps_0=\tfrac{1}{17}$
and consider the constants $\hat n_0$ from Lemma \ref{lem:hat-n0}, $\hat R$ from Proposition \ref{prop:R-hat}, and $\hat n_1=\hat n_1(\eps_0)$ from Lemma \ref{lem:m1/m'1-eps-hat}. We construct a sequence of primes recursively as follows:
\begin{equation} \label{eqn:cP-hat}
\begin{cases}
    \hat \nu_0=\text{ first prime greater than }  \max \left\{\hat R(1+\eps_0),\hat n_0,\hat n_1, \sqrt{18 \hat \t_6}\, \right \} , \\ 
    \hat \nu_{i+1}=\text{ first prime greater than } {\displaystyle \tfrac{\hat R \, \hat \n_i^2}{1-\eps_0}}\, .
\end{cases}
\end{equation}
We denote $\hat \cP=\{\hat \nu_i: i \ge 0\}$ this sequence of prime numbers. As before, the number $\hat R$ is very large, in particular we have $R  \gg \hat n_0, \hat n_1, \sqrt{18 \hat \t_6}$.

\medskip

Recall the construction of $(\k_0,\mu_0)$-bounded bases of type $\hat \e(k,\hat \cP)$: for each $k \in [1,\k_0]$ one chooses $m_k \in [k_0+1, \mu_0] \cap \ZZ$, $n_k \in \hat \cP$. These choices are made in increasing order, and we denote $\hat \e(k,\hat \cP)=\e^{n_k m_kk}$. Moreover, the most natural choice of $\mu_0$ is $\mu_0=2\k_0$.

\begin{proposition} \label{prop:non-existence-Y-1-hat}
Let $\hat \cP$ be the sequence of primes (\ref{eqn:cP-hat}). There does not exist a Kähler orbifold $Y$ with a pair of $(\k_0,\mu_0)$-bounded bases $\hat \e(k,\hat \cP)=\e^{n_km_kk}$, $\hat \e(k',\hat \cP)= \e^{n_{k'}m_{k'} k'}$.
\end{proposition}

\begin{proof}
Same proof as in Proposition \ref{prop:non-existence-Y-1}.
\end{proof}

\begin{proposition} \label{prop:non-existence-Y-2-hat}
Consider $\t_0=10 \, 290$, $\k_0=2 \t_0+2$, $\mu_0=2\k_0$, and 
\[
N \ge  \hat \nu_{\k_0} = \text{$\k_0$-th term of the sequence of primes $\hat \cP$.}
\]
There does not exist a Kähler orbifold $Y$ with $N^3$ orthogonal bases of type $\e^{nmk}$, with $1 \le n,m, k \le N$.
\end{proposition}

\begin{proof}
If there exists such an orbifold $Y$, then $Y$ satisfies all the properties from this section. Also, among the bases $\hat \e(m,\hat \cP)=\e^{n_km_kk}$, for $1 \le k \le \k_0$, constructed by choosing 
\[
n_2=\hat \nu_0 , \, \ldots, \,n_{\k_0}=\hat \nu_{\k_0}\, ,
\]
there would be at least two $(\k_0,\mu_0)$-bounded good bases. The non-existence of such an $Y$ follows from Proposition \ref{prop:non-existence-Y-1-hat}.
\end{proof}

The following follows immediately after all this work. This completes the proof 
of Theorem \ref{th:main-1}.

\begin{theorem} \label{th:main-sec7}
Consider $\t_0=10 \, 290$, $\k_0=2 \t_0+2$, $\mu_0=2\k_0$.
For a choice of $N_1 \ge \hat \nu_{\k_0}$, 
$N_2\geq \mu_0$, $K_0\geq \kappa_0$, the $K$-contact Smale-Barden manifolds $M$ constructed in Theorem \ref{th:main-0} do not admit any Sasakian structure.
\end{theorem}

\subsection{Estimation of the constants}

\begin{itemize}
    \item We have $\hat \t_0=10\,290$, the bound for the singular points of $Y$.
    \item We can take $\k_0=2 \hat \t_0+2=20 \, 582$ to ensure the existence of two $(\k_0,\m_0)$-bounded bases with $\m_0=2\k_0$.
    \item The number $\hat N_1=3 \k_0=61\,746$ is the lower bound for the value of $N$ in order to have two $(\k_0,2\k_0)$-bounded bases among the bases $\e^{nmk}$, $1 \le n,m,k \le N$.
\end{itemize}
Once given the number $\hat N_1$, we have the universal geometric constants orbifolds $Y$ with bases $\e^{nmk}$, $1 \le n,m,k \le N$, and with $N \ge \hat N_1$, given as follows:
\begin{itemize}
    \item $\hat \t_1=(18 \hat N_1^2-1)^2$ and $\hat \t_2=36 \hat N_1^2-1$ so that $-\hat \t_2 \le K_Y^2 \le \hat \t_1$.

    \item $\hat \t_4=144 \hat N_1^2+30$ so that $-C^2 \le 2g(C)+\hat \t_4$ for any nice curve $C \subset Y$.

    \item The number 
    \[
    \hat n_0=\left \lfloor \sqrt{\tfrac{36 \hat N_1^2-1+(18\mu_0^2-1)^2-72}{72}} \right \rfloor +1 ,
    \]
    to ensure $D_1^{nmk}$ is the positive curve for $n \ge \hat n_0$.
\end{itemize}
We also have:
\begin{itemize}
    \item The number $\hat \t_6=(18N_1^2-1)^2+(18 \mu_0^2+1)^2+(18 \k_0^2+1)^2$,
    from Lemma \ref{lem:bound-m1-hat}.
    \item The number $\hat n_1=18 \sqrt{\frac{\hat \t_6}{18}}=\sqrt{18 \hat \t_6}=320\,255\,755\,343$, from Lemma \ref{lem:m1/m'1-eps-hat}, with $\eps=\frac{1}{17}$.
    \item The number $\hat R=18^2 \lcm(2,3,4,\dots, \hat \t_4+18\mu_0^2+2)$, from Proposition \ref{prop:R-hat}.
\end{itemize}

If we take $\mu_0=65\,499$, $N_1=2 \mu_0-1=130\,997$, we get:
\begin{align*}
& \hat \t_1= 4\,709\,559\,895\,161\,746\, 824\,369 \, ,  && \hat \t_2=137\,252\,466\,575 \, , \\  & \hat \t_4=549\,009\,866\,334 \, ,
 && \hat n_0=3\,594\,524\,069  \, , \\ & \hat \t_6=
5\,697\,986\,046\,103\,220\,268\,099 \, , && \hat n_1=320\,255\,755\,343 \,.
\end{align*}

The number $\hat R$ is
\begin{align*}
& \hat R= 18^2 \lcm(2,3,4,\dots, 579\,510\,414\,464)= 18^2 \prod_{p \text{ prime}} p^{e_p} \, , \\
& \text{where $2 \le p \le 579\,510\,414\,464$, and $e_p=\lfloor \tfrac{\log(579\,510\,414\,464)}{\log p} \rfloor \leq \lfloor \tfrac{12}{\log p} \rfloor$} \, .
\end{align*}
The number $\hat R$ is by far the bigger number among those computed here. We can estimate easily that
\[
\hat R > 18^2 \cdot 210 \cdot 11^{4\, 118\,054\,809}
\]
which is again a huge number.
Note however that all these constants are much smaller than the corresponding numbers that we computed in Section \ref{sec:6} where no use was made from multiples of the third curve.

\subsection{Estimation for $\hat N$.}

Consider the sequence of primes $\hat \cP$ given by:
\[
\begin{cases}
    \hat \nu_0=\text{ first prime greater than } \frac{18}{17} \hat R \, , \\[5pt] 
    \hat \nu_{i+1}=\text{ first prime greater than } \frac{17}{16} \hat R \, \hat \n_i^2\, .
\end{cases}
\]
Let us give a lower and upper bound for $\hat N= \hat \nu_{\k_0}=\hat \nu_{20 \, 582}$.

\noindent {\bf Lower bound.}
As
$\hat \nu_{i+1} \geq \frac{17}{16} \hat R \, \hat \n_i^2$, we have by induction that 
\[
\hat \nu_k \ge \tfrac{16}{17} \left(\tfrac{18}{16}\right)^{2^k} \hat R^{2^{k+1}-1} \, .
\]
Hence 
\[
\hat N \ge \nu_{20\,582} \ge \tfrac{16}{17} \left(\tfrac{18}{16}\right)^{2^{20\,582}} \hat R^{2^{20\,583}-1} \, .
\]

\noindent{\bf Upper bound.}
Denote $\hat r>1$ a number so that for any $n \ge \hat R$ there exists a prime with $n<p<\hat r n$.
According to the result by Dusart \cite[Proposition 5.4]{Dusart}, 
 we can take $r=1+\log^{-3} (\hat R)$. Looking at the estimates above it is easy to see that $\log^{3} (\hat R)>10^{28}$, so we can take $\hat r=1+10^{-28}$.
Using this result, we get that 
$\hat \nu_0$ satisfies $\hat \nu_0 < \frac{18}{17} \hat r \hat R$,
and that 
$\hat \nu_{i+1} \leq \frac{17}{16} \hat r \hat R \n_i^2$. By induction we get
\[
\hat \nu_k \leq \tfrac{16}{17} \left(\tfrac{18}{16}\right)^{2^k} (\hat r \hat R)^{2^{k+1}-1} \, .
\]

\medskip

Summing up, to get a contradiction with the existence of a Sasakian structure, we need our orbifold $Y$ to have at least $\k_0=2\hat \t_0+2=20\,582$ bases of type $\hat \e(k,\hat \cP)$, hence we need a value of $\hat N$ greater or equal to $\nu_{20\,582}$. We obtain the following, result, which completes the proof of Theorem \ref{th:main-3}.

\begin{proposition}
Let $\hat N$ an integer such that
\[
\hat N \ge 
\tfrac{16}{17} \left(\tfrac{18}{16}\right)^{2^{20\,582}} (\hat r \hat R)^{2^{20\,583}-1} \, ,
\]
The K-contact Smale-Barden manifolds $M$ from Theorem \ref{th:main-0} with $N_1\ge \hat N$,
$N_2 \ge 41\,164$ and $K_0 \ge 20 \, 582$, cannot admit a Sasakian structure.
\end{proposition}

\subsection{Gaps in the torsion of $H_2(M,\ZZ)$.}

It is worthwhile to point out that in fact we can give many more examples of Smale-Barden manifolds which admit K-contact but not Sasakian structures. It suffices to notice that the non-existence of a Sasakian structure has been deduced solely from the existence of (at least) two suitable bases $\e=\e^{nmk}$, $\e=\e^{n'm'k'}$ for some values of $k,m,n$ satisfying some specific requirements. More precisely, in the proof of Theorem \ref{th:main-2} we need:
    \begin{itemize}
    \item the bases $\e^{nm1}$ for $(n,m) \in \{(i,i+1): 2 \le i \le 10\}$, to bound the number of singularities of $Y$.
    \medskip
    \item the bases $\e^{nm1}$ for the values $(n,m)$ with $1 \le m \le \mu_0=65 \, 500$, $\mu_0+1 \le n \le 2 \mu_0-1=N_1$, in order to ensure the existence of the universal geometric constants from Section \ref{sec:5}.
    \medskip
    \item the existence of two $\mu_0$-bounded good bases of type $\e^{\nu_i m 1}$, $\e^{\nu_j m' 1}$ with $1\le m, m' \le \mu_0$, and $\nu_i, \nu_j$ in the collection of disperse primes $\cP$. 
    \end{itemize}
    
 In the proof of Theorem \ref{th:main-3} we need:
 \begin{itemize}
 \item the bases $\e^{nmk}$ for the values $(n,m,k)=(3,2,1),
 (6,5,4),(9,8,7)$, to bound the number of singular points of $Y$.
 \medskip
 \item the bases $\e^{nmk}$ for the values $(n,m,k)=(2\k_0+k, \k_0+k, k)$, with $1\leq k\leq \k_0=20 \, 582$, to get the universal geometric bounds. 
 \medskip
     \item the existence of two $(\k_0,2\k_0)$-bounded bases of type $\e^{\hat \nu_i m k}$, $\e^{\hat \nu_j m' k'}$ with $1 \le k, k' \le \k_0$, $1\le m=\k_0+k, m'=\k_0+k' \le 2\k_0$, and $\hat \nu_i, \hat \nu_{j}$ in the collection of primes $\hat \cP$.
 \end{itemize}  
Finally, about the primes $p_{nmk}$. We only 
need that $p_{nmk}$ is $\geq 5$ and coprime with
$n,m,k$. 

In order to ensure these conditions, in Theorem \ref{th:main-2} (respectively in Theorem \ref{th:main-3}) we assumed that the orbifold $Y$ contained all the bases $\e^{nm1}$ (resp.\ the bases $\e^{nmk}$) for all the values of $n, m$ (resp.\ $n,m,k$) up to certain bounds. But we actually need much less, since it is enough that $Y$ has bases of type $\e^{nmk}$ where $n$ belongs to a very specific set of values, i.e.\ a sufficiently disperse sequence of primes $\cP$ and $\hat \cP$.
The sets $\cP$ and $\hat \cP$ were defined by using sequences $(x_k)_k$ of primes that satisfy a bound of type
$x_k \ge c \rho^{2^k}$
for some constants $c>0$, $\rho \gg 1$. For the sequence $\cP$ we have $c=\frac{16}{17} \frac{1}{R}$ and $\rho=\frac{18}{17}R^2$, while for $\hat\cP$ we have analogous constants changing $R$ by $\hat R$.
Lastly, recall that the sequences of primes $\cP$ and $ \hat \cP$ are not uniquely determined, since the only requirements is that the primes in the sequence are \emph{sufficiently apart} from each other. So in fact, there are many choices possible.

After this discussion, let us introduce some notation to properly enounce a sharper version of Theorems \ref{th:main-2} and \ref{th:main-3}.
Consider the set
 $$
 \mathcal{E}=\{(i,i+1):2\le i \le 10\} \cup \{(\mu_0+m,m): 1\leq m\leq \mu_0\} \cup \{(\nu_m,m): 1\leq m\leq \mu_0\}.
 $$
For each $e=(n_e,m_e) \in  \mathcal{E}$ take a prime $p_e$ subject to the condition that they are different,
$p_e\geq 5$, and $p_e$ coprime with $n_e,m_e$.

\begin{theorem} \label{thm:16}
Let $\cP=\{\nu_m\}_m$ be a sequence of primes constructed as in (\ref{eqn:cP}), $\mu_0 \ge 65 \, 500$, and let $\mathcal{E}$ the set defined above. Any Smale-Barden manifold $M$ such that
\[
H_2(M,\ZZ) = \ZZ^2 \oplus  \bigoplus_{e=(n_e,m_e) \in \mathcal{E}} \ZZ^{18n_e^2+2}_{p_{e}} \oplus \ZZ^{18m_e^2+2}_{p^2_{e}} \oplus \ZZ^{20}_{p^3_{e}} 
\]
cannot admit a Sasakian structure.
\end{theorem}
The manifold $M$ in Theorem \ref{thm:16} 
has torsion with as few as $3(2\mu_0+9)=393\, 027$ summands, which
coincides with the number of isotropy curves in the
orbifold $X'$.

Now, let us consider the set
 $$
 \cA=\{(3,2,1),
 (6,5,4),(9,8,7)\} \cup \{(2\k_0+k,\k_0+k,k): 1\leq k\leq \k_0\} \cup \{(\nu_k,\k_0+k,k): 1\leq k\leq \k_0\}.
 $$
For each $a=(n_a,m_a,k_a)\in \cA$, take a prime $p_{a}$
subject to the condition that they are different,
$p_{a} \geq 5$, and $p_{a}$ coprime with $n_a,m_a,k_a$.

\begin{theorem} \label{thm:17}
Let $\hat \cP=\{\hat \nu_k\}_k$ be a sequence of primes constructed as in (\ref{eqn:cP-hat}), 
$\k_0 = 20 \, 582$, and let $\cA$ the set defined above.
Any Smale-Barden manifold such that
\[
H_2(M,\ZZ) =
\ZZ^2 \oplus  \bigoplus_{a=(n_a,m_a,k_a)\in \cA}
\ZZ_{p_{a}}^{18n_a^2+2} \oplus \ZZ_{p^2_a}^{18m_a^2+2} \oplus \ZZ_{p^3_{a}}^{18k_a^2+2} 
\]
cannot admit a Sasakian structure.
\end{theorem}

The manifold $M$ in Theorem \ref{thm:17} 
has torsion with as few as $3(2\k_0+3)=123\, 501$ summands. Again this
coincides with the number of isotropy curves in the
orbifold $X'$.

Clearly, the K-contact manifolds from Theorem \ref{th:main-0} can be constructed so as to satisfy the hypothesis of any of Theorems \ref{thm:16} and \ref{thm:17}, because we have complete freedom to choose which of the curves among $T_n, T'_m, A_k$ we consider inside the isotropy locus of $X'$.

As a final remark, note that the hypothesis in Theorems \ref{thm:16} and \ref{thm:17} above allow the existence of huge gaps in the primes $p$ such that the $p$-torsion in $H_2(M,\ZZ)$ is non-zero.

\end{document}